\documentclass{article}
\usepackage{amssymb,amsmath,graphicx,ucs,enumerate,cite,hyperref}
\usepackage[utf8x]{inputenc}
\usepackage[english]{babel}

\newtheorem{definition}{Definition}
\newtheorem{conjecture}{Conjecture}
\newtheorem{theorem}{Theorem}
\newtheorem{proposition}{Proposition}
\newtheorem{lemma}{Lemma}
\newtheorem{corollary}{Corollary}

\newenvironment{proof}{\noindent \emph{Proof. }}{\hfill \hbox{\rlap{$\sqcap$}$\sqcup$}\\}

\title{When Periodicities Enforce Aperiodicity\thanks{This work was supported by the ANR project QuasiCool (ANR-12-JS02-011-01)}}
\author{
Nicolas Bédaride
\footnote{Aix Marseille Univ., CNRS, Centrale Marseille, I2M, UMR 7373, 13453 Marseille, France.}
\and
Thomas Fernique
\footnote{Univ. Paris 13, CNRS, Sorbonne Paris Cité, UMR 7030, 93430 Villetaneuse, France.}
}
\date{}
\begin{document}
\maketitle
\vspace{-5mm}

\begin{abstract}
Non-periodic tilings and local rules are commonly used to model the long range aperiodic order of quasicrystals and the finite-range energetic interactions that stabilize them.
This paper focuses on planar rhombus tilings, that are tilings of the Euclidean plane which can be seen as an approximation of a real plane embedded in a higher dimensional space.
Our main result is a characterization of the existence of local rules for such tilings when the embedding space is four-dimensional.
The proof is an interplay of algebra and geometry that makes use of the rational dependencies between the coordinates of the embedded plane.
We also apply this result to some cases in a higher dimensional embedding space, notably tilings with $n$-fold rotational symmetry.
\end{abstract}

\vspace{-5mm}
\tableofcontents

%%%%%%%%%%%%%%%%%%%%%%%%%%%%%%%%%%%%%%
\section{Introduction}

A {\em tiling} is a covering of some space by non-overlapping compact sets called {\em tiles}.
{\em Local rules} are constraints on the way neighbour tiles of a tiling can fit together.
Jigsaw puzzles provide a graphic example, with the dents and bumps as local rules.
Tilings and local rules only relatively recently went beyond recreational mathematics.
The first major step occured in the early 60's, when the logician Hao Wang asked in \cite{wang} if there exists an algorithm which decides whether any given finite set of tiles can tile the whole plane (each tile can be used several times).
His student Robert Berger gave a negative answer in \cite{berger}, with a key ingredient of his proof being the first-ever tile set which can tile the plane but only in a non-periodic way (that is, no tiling is invariant by a non-trivial translation).
Such tile sets are said to be {\em aperiodic}.
Some other examples were since then discovered, with the most celebrated one probably being the Penrose tiles \cite{penrose2}.
The second major step occured twenty years later, with the ground-breaking discovery by Dan Shechtman of {\em quasicrystals}, that are non-periodic but nevertheless ordered materials \cite{shechtman}.
The link with (ordered) non-periodic tilings was indeed quickly done, with local rules modeling the energetic finite-range interactions between atoms \cite{LS}.
A primordial issue in mathematical physics then became to determine which type of quasicrystalline structure can exist.\\

In this paper, we focus on the case of the rhombus tilings of the plane.
Such tilings can indeed profitably be seen as digitizations of surfaces in higher dimensional spaces.
Among them are the plane digitizations obtained by the so-called {\em canonical cut and projection} method (see \cite{debruijn,GR}).
They are said to be {\em planar} and aim to model the long-range order of quasicrystals.
This connects the algebraic parameters of a plane with the geometry of its digitization, and one of main issues is: which planes have a digitization characterized by local rules?
Note that rhombus tilings are surely far from comprising all the existing tilings, but they nevertheless provide a large family which can notably model all the quasicrystalline symmetries yet experimentally observed, with the exception of the icosahedral one (which requires rhombohedra tilings of the three-dimensional space; we actually checked that our method extends to this specific case).\\

We also focus on a special kind of local rules, namely {\em uncolored weak} ones.
Formally, local rules can be expressed as a finite set of finite patterns that must be avoided ({\em forbidden patterns}).
They are {\em colored} when the same tile can appear in different colors in the forbidden patterns, thus playing different roles.
Colored local rules are more powerful but also less realistic from the physical viewpoint (the model is even more complicated).
This could explain why uncolored ones have retained the attention of many authors.
{\em Weak} local rules have been introduced for  planar rhombus tilings by Leonid Levitov in \cite{levitov}, as opposed to {\em strong} local rules.
Whereas strong local rules enforce a tiling to be a specific digitization of a plane, weak ones only require the digitization to stay at bounded distance from a plane.
In other words, weak local rules allow short-range disorder (some authors speak about {\em bounded perp-space fluctuations}).\\

So, which planar rhombus tilings do admit weak uncolored local rules?
This has been an issue of much debate in mathematical physics during the early 90's (see, {\em e.g.}, \cite{burkov,levitov, le92,le92b,le92c,le93,le95,le95b,socolar}).
Several conditions have been found, usually stated in terms of algebraic properties of the digitized plane, but no complete characterization has yet emerged.
At a minimum, examples of non-periodic tilings do exist ({\em e.g.}, the Penrose tilings).
At the other extreme, Thang Le has proved in \cite{le95} that digitized plane admitting weak uncolored local rules are always generated by vectors whose entries are algebraic numbers.
Let us mention, in comparison, that a plane digitization is proven in \cite{FS} to admit weak {\em colored} local rules if and only if the plane can be generated by {\em computable} vectors (that is, their entries can be computed to within any desired precision by a finite, terminating algorithm).\\

Our results are along those lines.
We rely on the notion of {\em subperiod}, which is related to the {\em SI-condition} introduced by Leonid Levitov in \cite{levitov}.
The idea is that the vectors which generate an irrational plane can nevertheless have rational depencies between their entries.
The subperiods catch such kind of dependencies, which turn out to easily translate into equations on the Grassmann coordinates of the plane, yielding a system of polynomial equations.
The point is that the subperiods can be enforced by specific local rules.
Our main result, roughly stated, is that the existence of such specific local rules is equivalent to the zero-dimensionality of the corresponding system of polynomial equations (at least when the digitized plane lives in $\mathbb{R}^4$, Corollary~\ref{cor:local_rules_codim2}).
Actually, a fruitful approach turned out to isolate the subquestion of planarity: when do the local rules associated with the subperiods of a plane allow only rhombus tilings which are digitizations of (any) planes?
Theorem~\ref{th:planarity_codim2} gives an answer when, again, the plane lives in $\mathbb{R}^4$.
This is further used in higher dimensional cases (Proposition \ref{prop:n_fold_planarity} and Corollary~\ref{cor:local_rules_n_fold}).
The general goal of reducing the existence of local rules to the resolution of a system of equations remains to be achieved.\\

The paper is organized as follows.
Section \ref{sec:settings} gives formal definitions of the above mentioned notions (planar rhombus tilings, local rules, subperiods\ldots).
In Section \ref{sec:codim2}, we consider plane digitization in $\mathbb{R}^4$, we state and prove the main result (Theorem~\ref{th:planarity_codim2}), as well as provide illustrative examples.
We show in Section \ref{sec:higher_codim} how this result can, under some conditions, be extended to higher dimensional spaces.
In particular, we apply our method to so-called {\em $n$-fold tilings}, which play a prominent role in modelling quasicrystals.
We show that there are local rules when $n$ is an odd multiple of $5$ or $7$ (Corollary~\ref{cor:local_rules_n_fold}).
Although it has been already proven by Joshua Socolar \cite{socolar} that local rules exist for any odd $n$, our proof is more algebraic and does not rely on the specific geometry of $n$-fold tilings.
Moreover, we also get some new cases if we in addition allow a {\em minimization principle} (Proposition~\ref{prop:even_fold_minimization}).

%%%%%%%%%%%%%%%%%%%%%%%%%%%%%%%%%%%%%%
\section{Settings}
\label{sec:settings}

%%%%%%%%%%%%%%%%%%%%%%%%%%%%%%%%%%%%%%
\subsection{Planar rhombus tilings}
\label{sec:rhombus_tilings}

Let $\vec{v}_1,\ldots,\vec{v}_n$ be $n\geq 3$ pairwise non-collinear unitar vectors of the Euclidean plane.
They define the $\binom{n}{2}$ rhombus {\em prototiles}
$$
T_{ij}=\{\lambda\vec{v}_i+\mu\vec{v}_j~|~0\leq\lambda,\mu\leq 1\}.
$$
A {\em tile} is a translated prototile (tile rotation or reflection are forbidden).
A {\em rhombus tiling} is a covering of the Euclidean plane by interior-disjoint tiles satisfying the {\em edge-to-edge} condition: whenever the intersection of two tiles is not empty, it is either a vertex or an entire edge.
It is said to be {\em non-degenerated} if each tile $T_{ij}$ appears at least one time.
A {\em pattern} is a finite (usually connected) union of tiles which appears in some tiling.\\

Let $\vec{e}_1,\ldots,\vec{e}_n$ be the canonical basis of $\mathbb{R}^n$.
Following Levitov \cite{levitov}, a rhombus tiling is {\em lifted} in $\mathbb{R}^n$ as follows: an arbitrary vertex is first mapped onto the origin $\mathbb{R}^n$, then each tile $T_{ij}$ is mapped onto the $2$-dimensional face of a unit hypercube of $\mathbb{Z}^n$ generated by $\vec{e}_i$ and $\vec{e}_j$, with two tiles adjacent along an edge $\vec{v}_i$ being mapped onto two faces adjacent along an edge $\vec{e}_i$.
This lifts the boundary of a tile -- and by induction the boundary of any patch of tiles -- onto a closed curve of $\mathbb{R}^n$ and hence ensures that the image of a tiling vertex do not depends on the path followed to get from the origin to this vertex.
The lift of a tilings is thus a ``stepped'' surface of codimension $n-2$ in $\mathbb{R}^n$ (unique up to the choice of the initial vertex).
By extension, such a rhombus tiling is said to have codimension $n-2$.\\

The lift is the graph of a function from $\mathbb{R}^2$ to $\mathbb{R}^n$ which is Lipschitz continuous, with a Lipschitz constant that can be chosen to depend only on the $\vec{v}_i$'s.
Indeed, the limit in how fast this function can change between two points in a tile depends only on the way this tile is lifted in $\mathbb{R}^n$, and this then extends to any two points of the tiling.
Given the tiles, the set of lifts of all the possible tilings of the plane are thus uniformly Lipschitz continuous.\\

A rhombus tiling is said to be {\em planar} if there is $t\geq 1$ and an affine plane $E\subset \mathbb{R}^n$ such that the tiling can be lifted into the tube $E+[0,t]^n$ (we need $t\geq 1$ to have tiles into the tube).
The smallest suitable $t$ is called the {\em thickness} of the tiling, and the corresponding $E$ is called the {\em slope} of the tiling.
Both are uniquely defined.
Following Levitov \cite{levitov}, one speaks about strong or weak planarity depending on whether $t=1$ or $t>1$.
A planar rhombus tiling is thus an {\em approximation} of its slope: the less the thickness, the better the approximation.\\

Figure~\ref{fig:codim1} illustrates this in the codimension one case.
Examples in higher codimensions shall be further provided.\\

\begin{figure}[hbtp]
\includegraphics[width=0.31\textwidth]{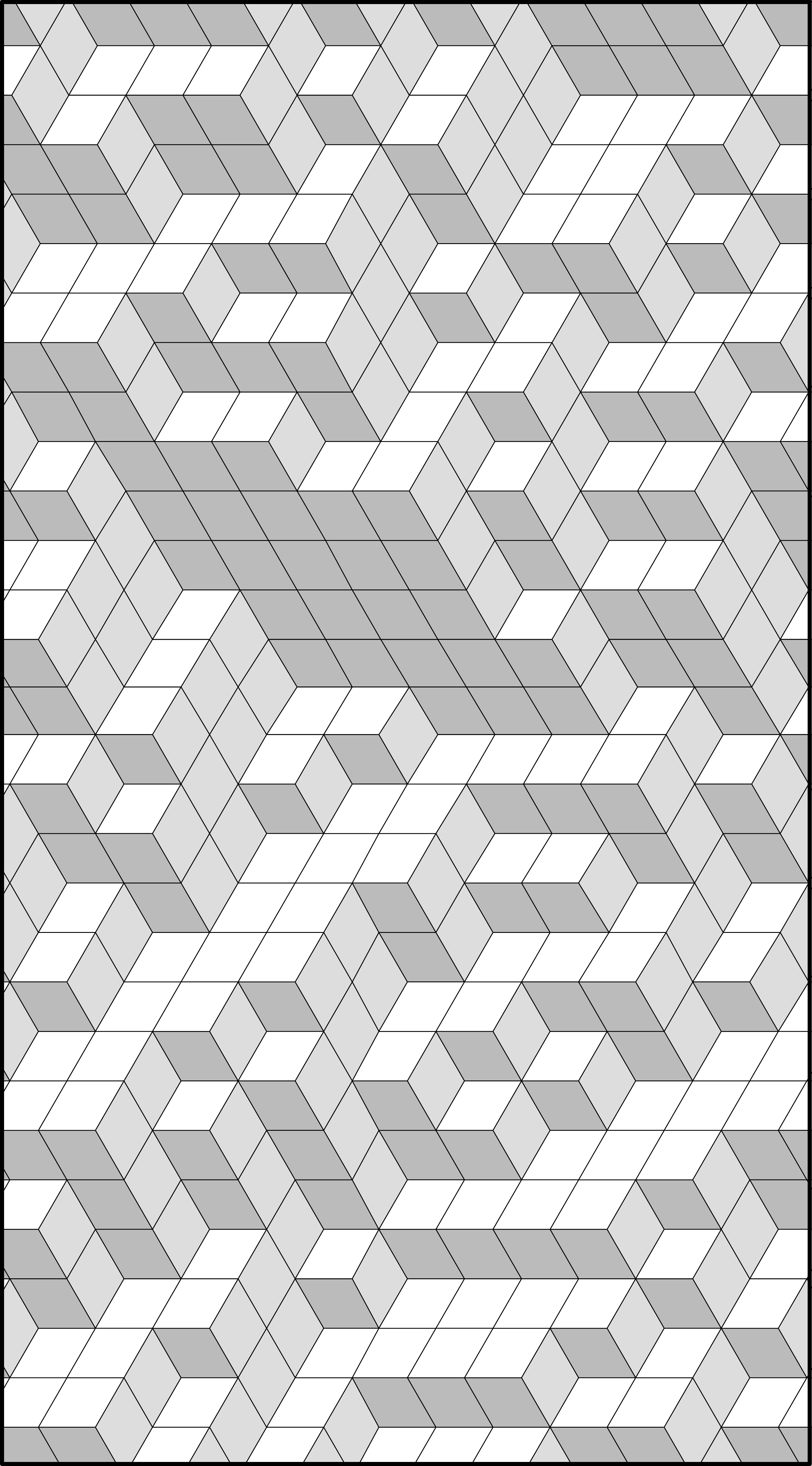}
\hfill
\includegraphics[width=0.31\textwidth]{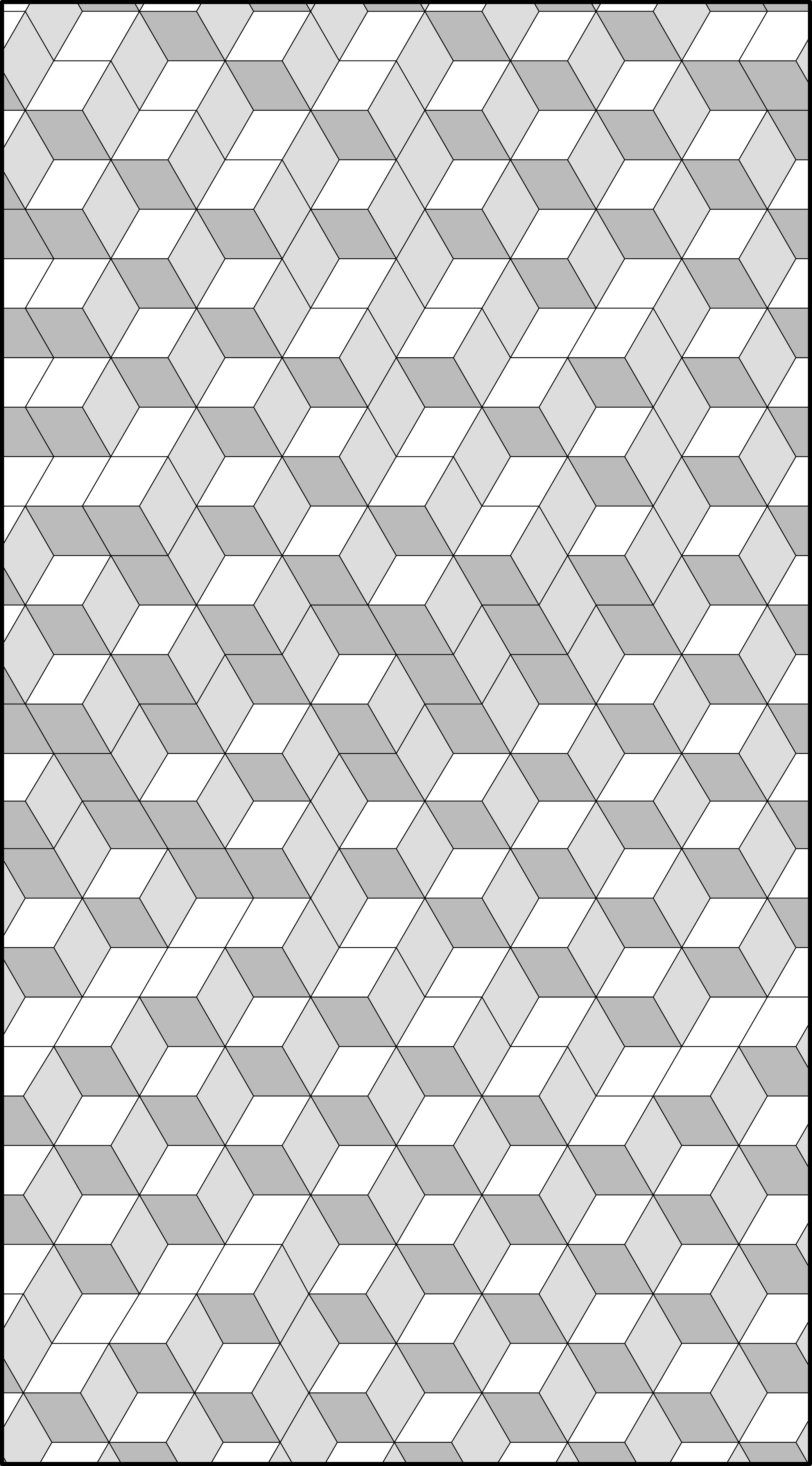}
\hfill
\includegraphics[width=0.31\textwidth]{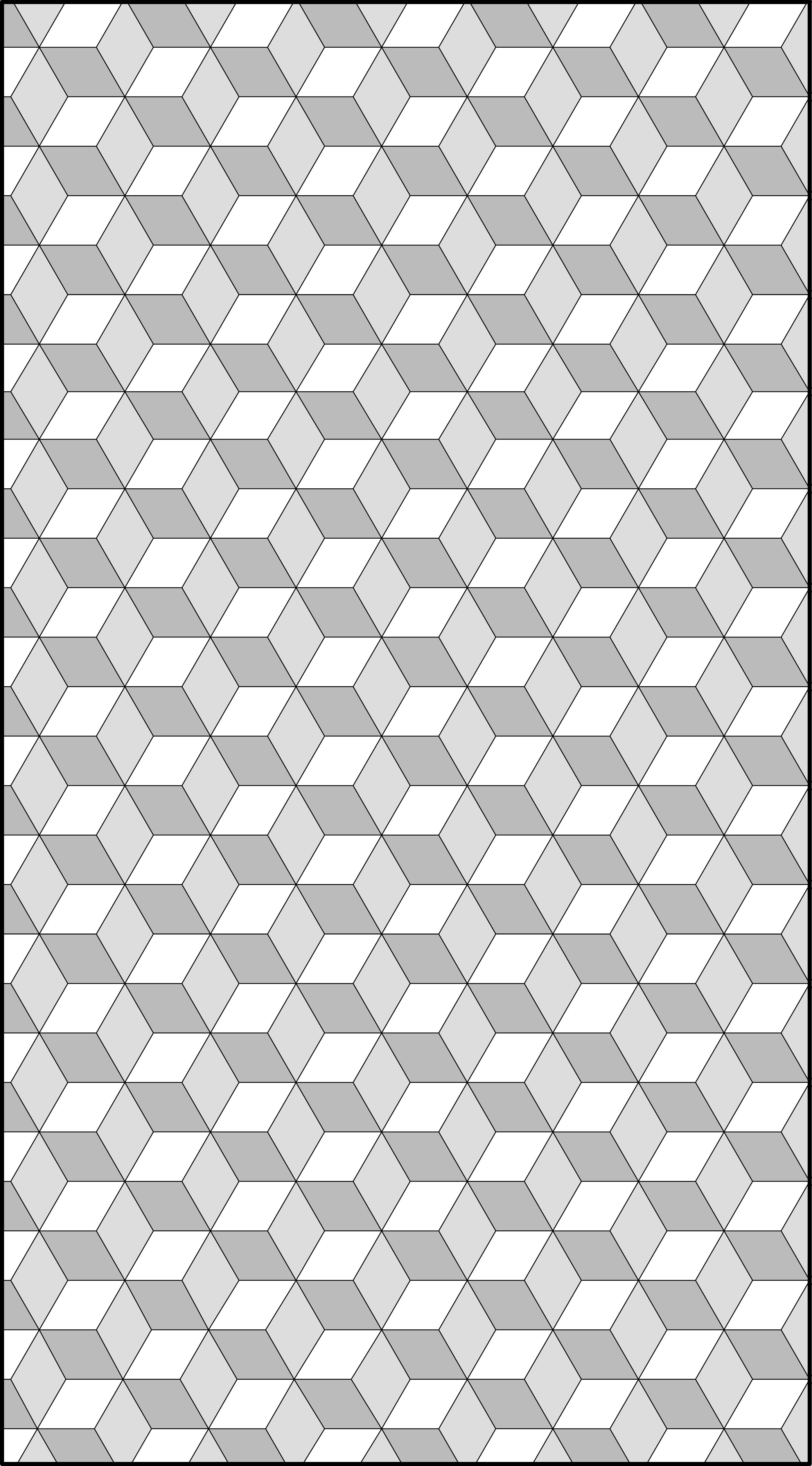}
\caption{
Planar codimension one rhombus tiling with decreasing thickness (from left to right).
Tile are colored to help to visualize the lift in $\mathbb{R}^3$.}
\label{fig:codim1}
\end{figure}

Strongly planar rhombus tilings are also referred to as {\em canonical cut and project tilings}.
They are {\em uniformly recurrent}, that is, whenever a pattern occurs once, there exists $R\geq 0$ such that this pattern reoccurs at distance at most $R$ from any point of the tiling.
Weakly planar rhombus tilings are not necessarily uniformly recurrent.
Nevertheless, in any planar rhombus tiling, the ratio of a given prototile among the prototiles occuring at distance at most $R$ from a point of the tiling admits a limit when $R$ goes to infinity, called its {\em frequency}, which depends only on the slope (see Prop.~\ref{prop:frequences}).

%%%%%%%%%%%%%%%%%%%%%%%%%%%%%%%%%%%%%%
\subsection{Local rules}
\label{sec:local_rules}

Draw a disk of diameter $r$ on a tiling and consider the pattern formed by all the tiles which intersect this disk: this is called a {\em $r$-map} of the tiling.
The {\em $r$-atlas} of a tiling is the set of all its $r$-maps.
We use this to define the weak uncolored local rules mentioned in the introduction, that we shall simply refer to as local rules since they are the only type further considered:

\begin{definition}\label{def:local_rules}
A strongly planar rhombus tiling of slope $E$ is said to admit {\em local rules} of {\em diameter} $r$ and {\em thickness} $t$ if, whenever its $r$-atlas contains the $r$-atlas of another rhombus tiling, this latter is planar with slope $E$ and thickness at most $t$.
By extension, the slope $E$ itself is said to admit local rules.
\end{definition}

When a tiling admits local rules of diameter $r$, the patterns of the $r$-atlas are themselves called local rules.
Since $r$-atlas of rhombus tilings are finite, it is equivalent and sometimes more convenient to define local rules by giving a set of patterns which are not allowed to occur in these local rules.
These patterns are said to be {\em forbidden}.\\

As for planar rhombus tilings, one speaks about strong or weak local rules depending on whether $t=1$ or $t>1$.
This paper aims to characterize the slopes which admit local rules.
Before focusing on totally irrational slopes, let us first dispose of the matter on slopes which contain rational directions.
Consider, first, the case of a {\em rational slope}:

\begin{proposition}\label{prop:ratio2}
A slope with two rational directions admits strong local rules.
\end{proposition}

\begin{proof}
To each rational direction of a slope corresponds a {\em period} of the corresponding strongly planar tilings, that is, a translation vector which leaves them invariant.
If there are two (independent) such directions, then the tilings have a bounded fundamental domain and it suffices to consider local rules whose dia\-me\-ter is greater than the one of this fundamental domain.
\end{proof}

\noindent Consider, now, the case of a slope which is neither rational nor irrational:

\begin{proposition}\label{prop:ratio1}
A slope with exactly one rational direction admits no local rules.
\end{proposition}

\begin{proof}
Let $\mathcal{T}$ be a strongly planar tiling with exactly one rational direction.
Let $r>0$ be given.
Consider a period $\vec{p}$ of $\mathcal{T}$ and consider the pattern $S$ formed by the tiles at distance less than $r$ from the segment $\vec{p}$.
Since $\mathcal{T}$ is uniformly recurrent, there exists $\vec{q}\neq\vec{0}$ such that a translation by $\vec{q}$ maps $S$ onto one of its reoccurrences.
This yields two periodic parallel and equal ``sticks'' respectively formed by the tiles at distance less than $r$ from $\mathbb{R}\vec{p}$ and $\mathbb{R}\vec{p}+\vec{q}$.
Consider now the patterns $T_\lambda$ formed by the tiles at distance less than $r$ of the segment joining $\lambda\vec{p}$ and $\lambda\vec{p}+\vec{q}$, for $\lambda\in\mathbb{R}$.
Since there is a finite number of different patterns of a given size, there are $\lambda_1$ and $\lambda_2$ such that $T_{\lambda_1}=T_{\lambda_2}$.
Consider now the tiling $\mathcal{T}'$ with fundamental domain the parallelogram with vertices $\lambda_1\vec{p}$, $\lambda_1\vec{p}+\vec{q}$,  $\lambda_2\vec{p}+\vec{q}$ and  $\lambda_2\vec{p}$.
It has two rational directions and thus cannot have the slope of $\mathcal{T}$.
However, by construction, any $r$-map of $\mathcal{T}'$ is also a $r$-map of $\mathcal{T}$.
This shows that $\mathcal{T}$ does not admit local rules of any diameter $r$.
\end{proof}

\noindent The case on which we shall focus is thus the one of irrational slopes.

%%%%%%%%%%%%%%%%%%%%%%%%%%%%%%%%%%%%%%
\subsection{Subperiods}
\label{sec:subperiods}

Let us introduce this central notion:

\begin{definition}\label{def:subperiod}
The {\em $ijk$-shadow} of a rhombus tiling is the orthogonal projection of its lift onto the space generated by $\vec{e}_i$, $\vec{e}_j$ and $\vec{e}_k$.
An {\em $ijk$-subperiod} of a rhombus tiling is a prime period
of its $ijk$-shadow, hence an integer vector.
A {\em lift} of such a subperiod is any vector of $\mathbb{R}^n$ which projects on it in the $ijk$-shadow.
\end{definition}

A rhombus tiling has thus $\binom{n}{3}$ shadows, which are codimension one surfaces in $\mathbb{R}^3$.
By extension, we call subperiods of a slope the subperiods of the strongly planar rhombus tiling with this slope.
Figure~\ref{fig:subperiods} illustrates the notion of subperiod, while Figure~\ref{fig:forcing_subperiods} illustrates the following proposition.\\

\begin{figure}[hbtp]
\includegraphics[width=0.31\textwidth]{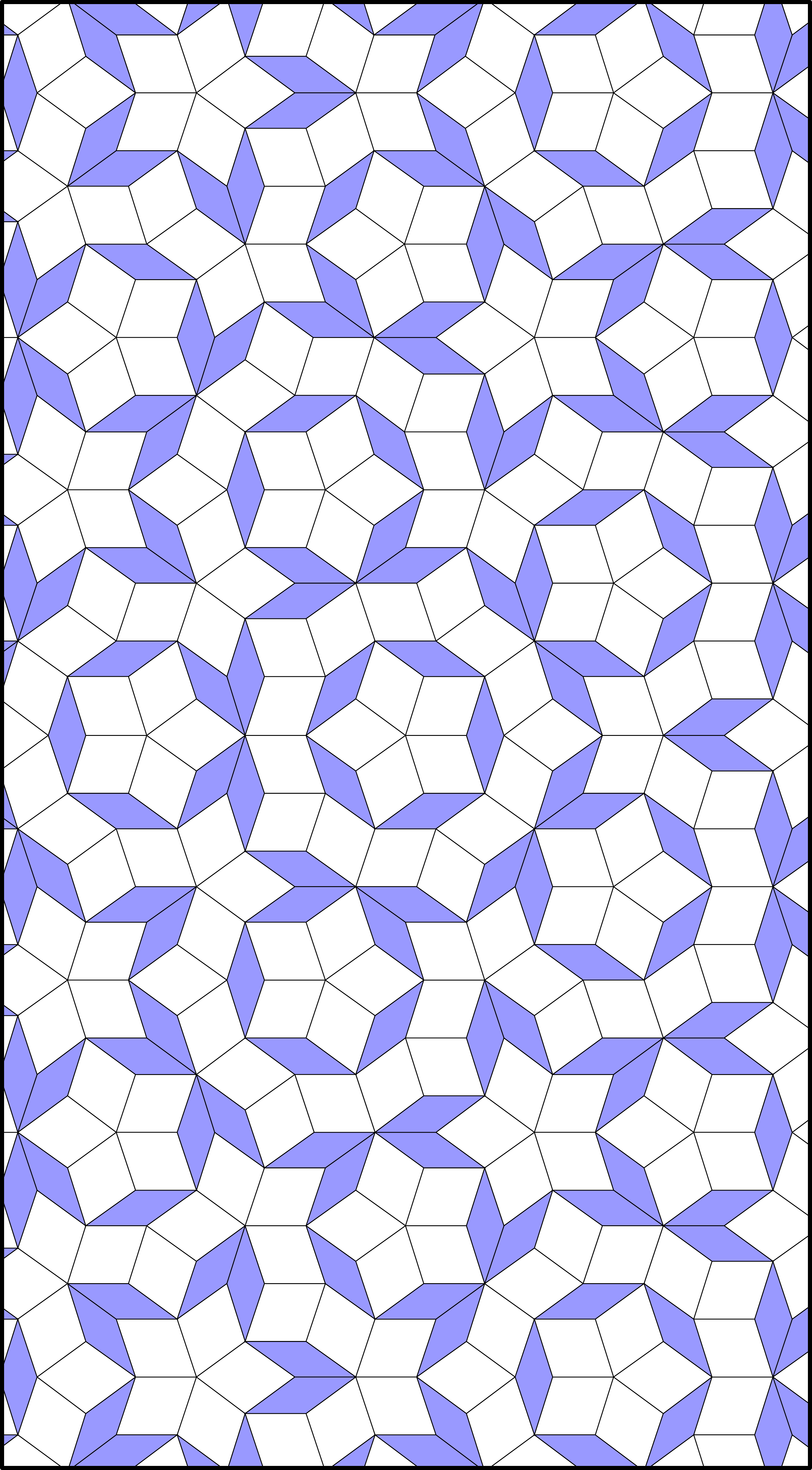}
\hfill
\includegraphics[width=0.31\textwidth]{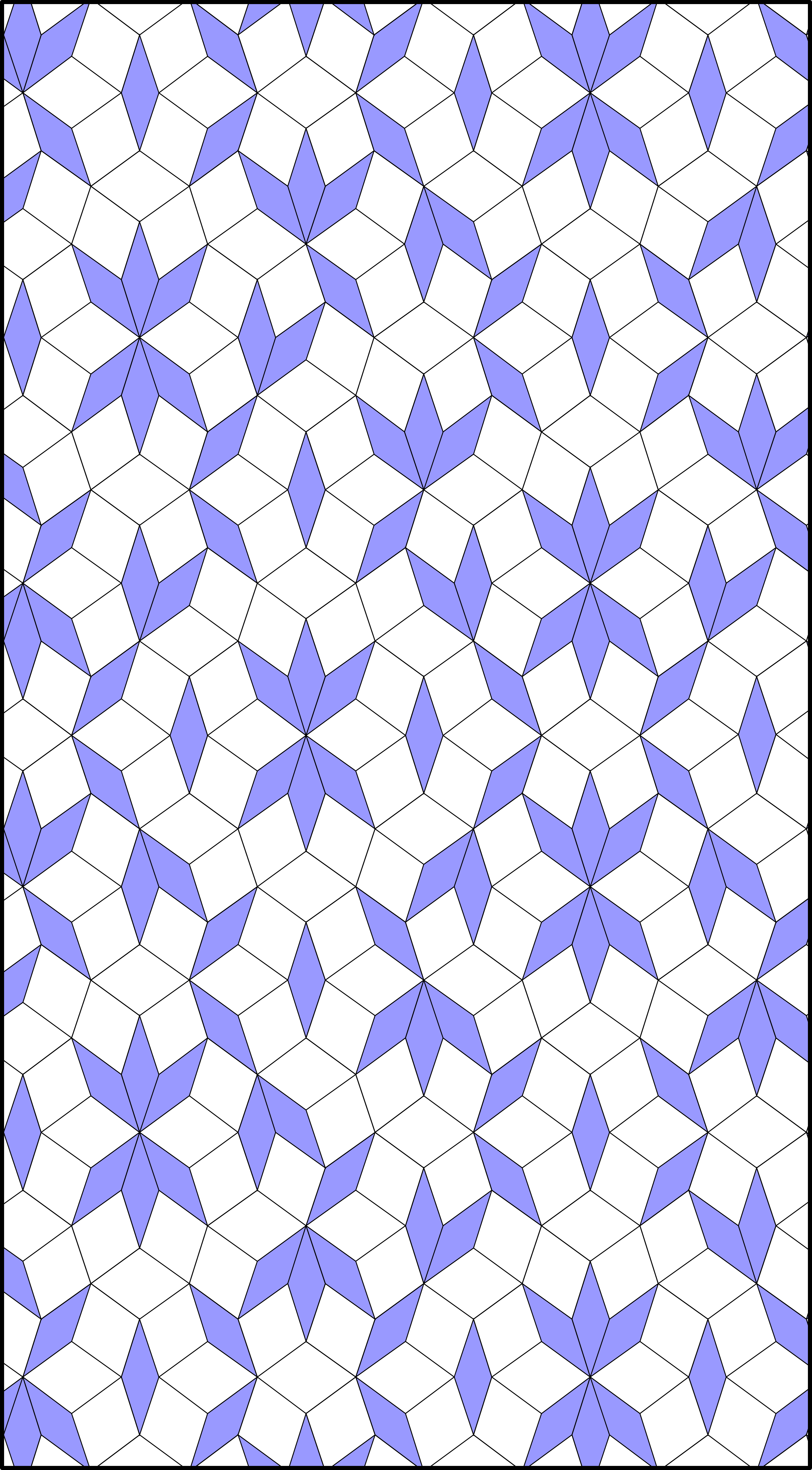}
\hfill
\includegraphics[width=0.31\textwidth]{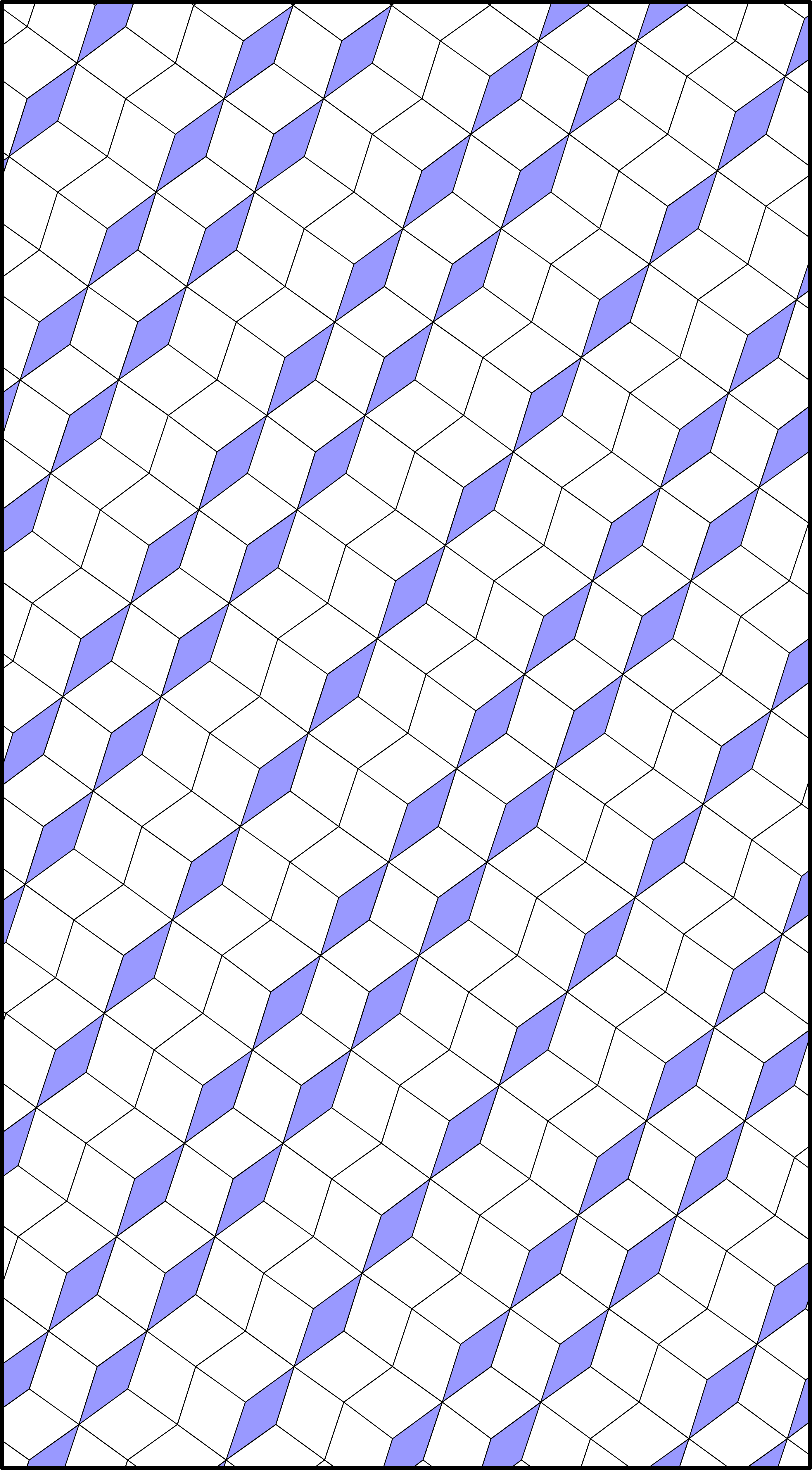}
\caption{A codim. $3$ tiling (left; actually it is a Penrose tiling). By orthogonally projecting along a basis vector of $\mathbb{R}^5$ we get a codim. $2$ tiling (center).
By orthogonally projecting along a second basis vector we get a codim. $1$ tiling (right) which is periodic: its prime period is a subperiod of the previous tilings.}
\label{fig:subperiods}
\end{figure}

\begin{proposition}\label{prop:subperiod_rules}
The subperiods of a slope can be enforced by local rules.
\end{proposition}

\begin{proof}
Let $\vec{p}\in\mathbb{R}^3$ be a subperiod of a slope $E\subset\mathbb{R}^n$ and $\pi$ the orthogonal projection on basis vectors such that $\vec{p}\in\pi(E)$.
The union $\mathcal{A}$ of the $r$-atlases of strongly planar rhombus tilings of slope $\pi(E)$ enforce their $\vec{p}$-periodicity as soon as $r\geq ||\vec{p}||$.
Then, the uniform recurrence of the strongly planar rhombus tilings of slope $E$ ensures that there is $R\geq 0$ such that the image under $\pi$ of the union $\mathcal{B}$ of their $R$-atlases contains all the patterns of $\mathcal{A}$.
Now, if a rhombus tiling has a $R$-atlas included in $\mathcal{B}$, then its image under $\pi$ has a $r$-atlas included in $\mathcal{A}$ and thus admits $\vec{p}$ as a period.
Hence, by definition, the initial tiling admits $\vec{p}$ as a subperiod (enforced by local rules of diameter $R$).
\end{proof}

\begin{figure}[hbtp]
\includegraphics[width=0.31\textwidth]{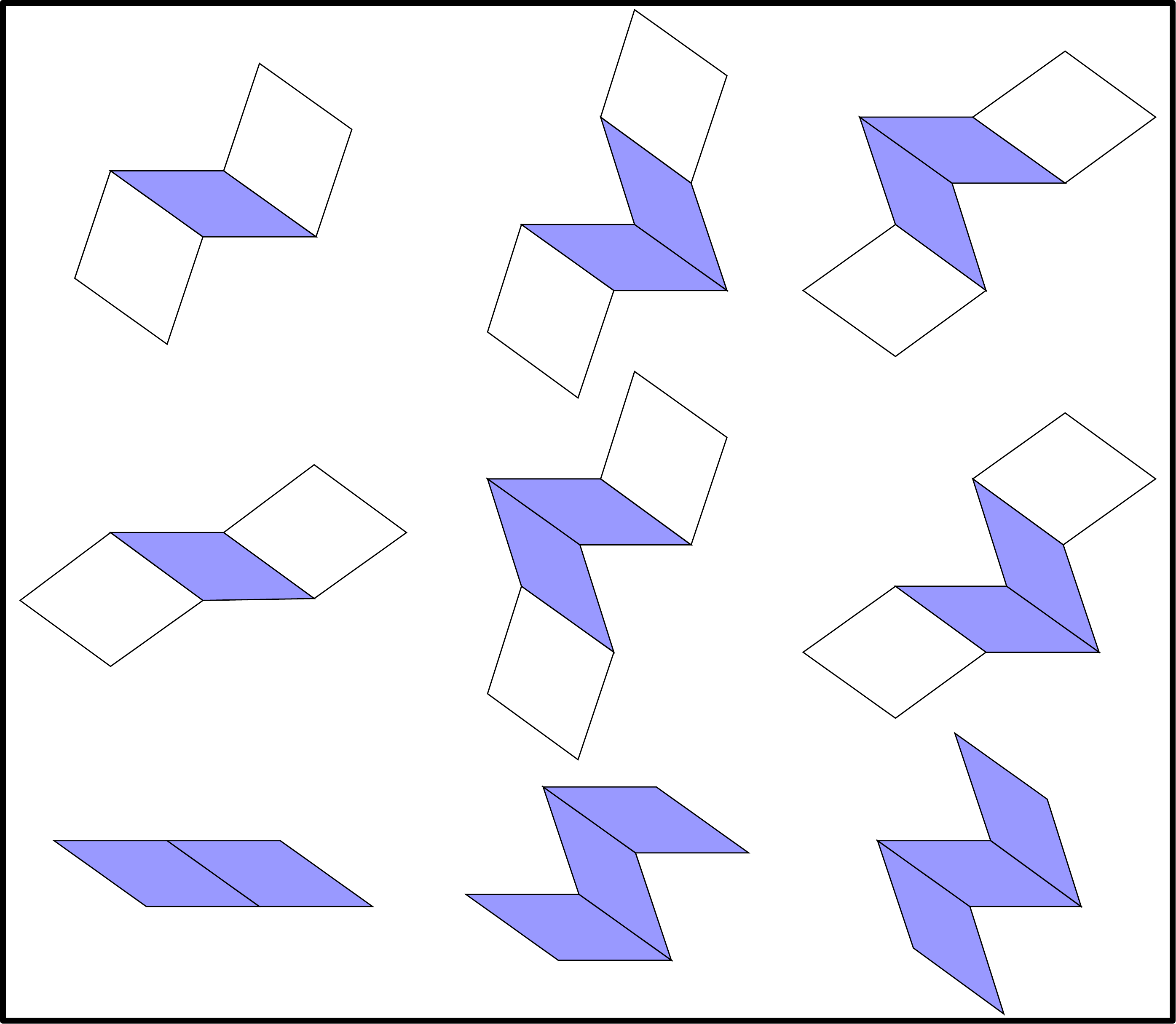}
\hfill
\includegraphics[width=0.31\textwidth]{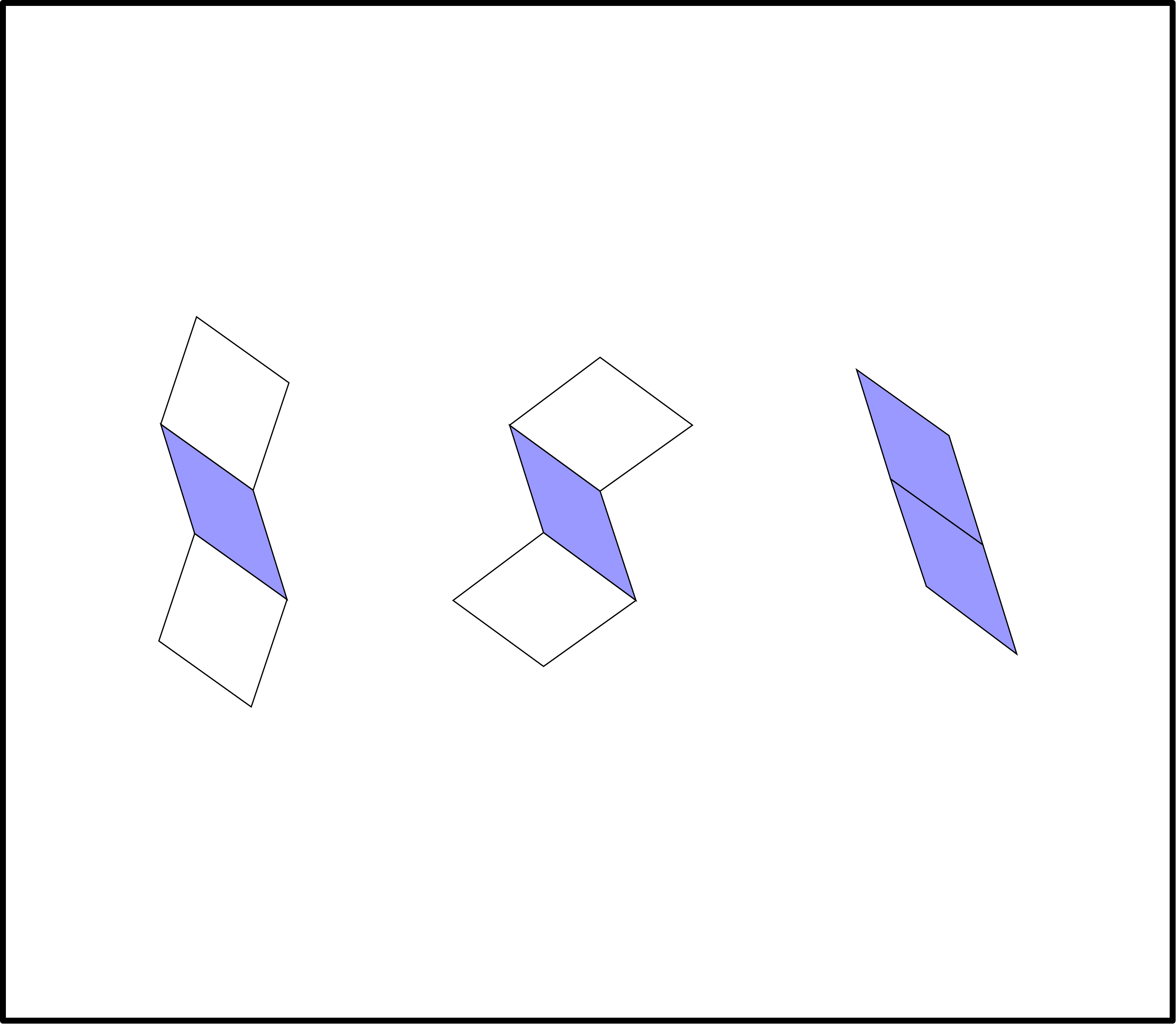}
\hfill
\includegraphics[width=0.31\textwidth]{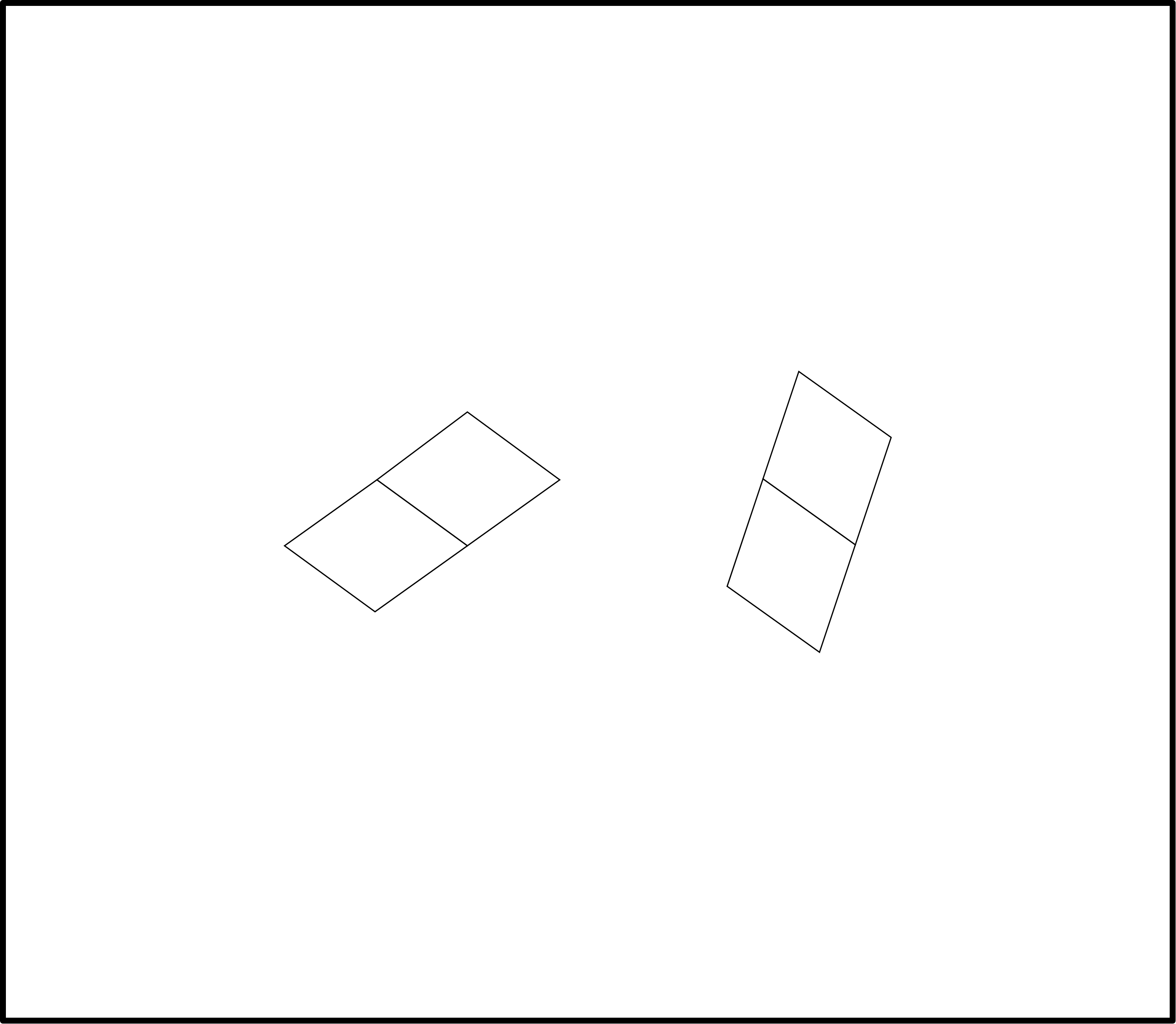}
\caption{
A codim. $1$ rhombus tiling without any occurence of the two forbidden patterns in the rightmost box must have the same period as the rightmost tiling on Fig.~\ref{fig:subperiods}.
Both tilings can however greatly differ.
In particular, these two forbidden patterns do not enforce planarity.
Then, a codim. $2$ (resp. codim. $3$) rhombus tiling which, in addition, avoids the forbidden patterns in the central box (resp. in all the boxes) must have the same subperiod as the central tiling (resp. the rightmost tiling) on Fig.~\ref{fig:subperiods}.
Not all the tilings with the same subperiod as the tilings on Fig.~\ref{fig:subperiods} are allowed, but at least these latter do.
}
\label{fig:forcing_subperiods}
\end{figure}

%%%%%%%%%%%%%%%%%%%%%%%%%%%%%%%%%%%%%%
\subsection{Grassmann coordinates}
\label{sec:grassmanniennes}

Let us recall the notion of {\em Grassmann coordinates} in our particular case (for a general presentation, see, {\em e.g.}, \cite{HoP}, Chap.~7).
The Grassmann coordinates of a plane $E$ generated by $(u_1,\ldots,u_n)$ and $(v_1,\ldots,v_n)$ are the $\binom{n}{2}$ real numbers
$$
G_{ij}:=u_iv_j-u_jv_i,
$$
for $i<j$.
We write $E=(G_{ij})_{i<j}$, with the Grassmann coordinates being ordered by lexicographic order on their indices.
Grassmann coordinates are defined up to a common multiplicative factor and turn out to not depend on the choice of the generating vectors.
They are moreover characterized: a non-zero $\binom{n}{2}$-tuple of reals are the Grassmann coordinates of some plane if and only if they satisfy the $\binom{n}{4}$ following quadratic equations, called Plücker relations:
$$
G_{ij}G_{kl}=G_{ik}G_{jl}-G_{il}G_{jk},
$$
for $i<j<k<l$.
By extension, we call Grassmann coordinates of a planar rhombus tiling the Grassmann coordinates of its slope; they can actually be ``read'' on the tiles:

\begin{proposition}\label{prop:frequences}
%The ratio of $T_{ij}$ to $T_{kl}$ in a planar rhombus tiling is $|G_{ij}/G_{kl}|$.
The frequency of $T_{ij}$ in a planar rhombus tiling is $\frac{|G_{ij}|}{\sum_{k<l}|G_{kl}|}$.
\end{proposition}

In particular, if $E$ has a zero Grassmann coordinate $G_{ij}$ then the tile $T_{ij}$ does not appears in planar tilings of slope $E$, that is, those are degenerated tilings (we shall avoid this case further).
Note also that the sign of a Grassmann coordinate of a planar tiling depends only on the $\vec{v}_i$'s (a slope with a different sign would yield a tiling which do not project correctly onto a tiling of the plane).
The proof of the above proposition, further not used, is left to the reader.
We will rather rely on the following:

\begin{proposition}\label{prop:subperiod_grassmann}
If a planar rhombus tiling has an $ijk$-subperiod $(p,q,r)$, then
$$
pG_{jk}-qG_{ik}+rG_{ij}=0.
$$
\end{proposition}

\begin{proof}
Consider the $ijk$-shadow of a planar rhombus tiling.
It is a planar rhombus tiling in $\mathbb{R}^3$ whose slope is generated by $(u_i,u_j,u_k)$ and $(v_i,v_j,v_k)$, hence has normal vector $(G_{jk},-G_{ik},G_{ij})$.
This vector thus has zero dot product with any vector in the slope, in particular with $(p,q,r)$: this yields the claimed relation.
\end{proof}

To each subperiod thus corresponds a linear relation with integer coefficients on Grassmann coordinates.
Together with the Plücker relations, this yields a system of polynomial equations.
If this system has a unique solution, then subperiods -- hence local rules by Prop.~\ref{prop:subperiod_rules} -- can enforce planar rhombus tiling to have this solution as slope.
Actually, this remains true if there are finitely many solutions, {\em i.e.}, if the system of polynomial equations is zero-dimensional, because one can always increase the diameter of local rules to select one among finitely many slopes.
One can then use very efficient algorithms (usually relying on Gröbner bases) to determine whether this system is zero-dimensional.
However, in order to conclude that such a slope has local rules, it must be proven that local rules can also enforce the {\em planarity} itself: this becomes the key issue.

%%%%%%%%%%%%%%%%%%%%%%%%%%%%%%%%%%%%%%
\section{Codimension two}
\label{sec:codim2}

%%%%%%%%%%%%%%%%%%%%%%%%%%%%%%%%%%%%%%
\subsection{Statement of the main result}
\label{sec:statement_codim2}

Subperiods are said to {\em enforce planarity} when the rhombus tilings with all these subperiods are planar with a uniformly bounded thickness.
Let us stress that there is no restriction on the number of slopes: it can be infinite.
The planarity is said to be {\em irrational} if there is at least one planar rhombus tilings with an irrational non-degenerated slope which has these subperiods (otherwise, we refer to Prop.~\ref{prop:ratio2} and \ref{prop:ratio1}).\\

We here focus on codimension two rhombus tilings, as the most simple non-trivial case.
Codimension one tilings are indeed trivial: any subperiod is a period, whence the only slopes that can be enforced by subperiods are rational ones (according to Prop.~\ref{prop:ratio2}).
Higher codimension tilings shall be considered in the next section.
The main result we get is the following:

\begin{theorem}\label{th:planarity_codim2}
The subperiods of a codimension two rhombus tiling enforce irrational planarity if and only if three of them, each in a shadow with only one period, can be lifted in an irrational non-degenerated plane onto pairwise non-collinear vectors.
This holds when subperiods characterize finitely many slopes.
\end{theorem}

\noindent With Prop.~\ref{prop:ratio2} and \ref{prop:subperiod_rules}, this easily yields

\begin{corollary}\label{cor:local_rules_codim2}
If a codimension two planar rhombus tiling has subperiods which characterize finitely many slopes, then it admits local rules.
\end{corollary}

This sufficient condition can be algorithmically checked on a given slope $E\subset\mathbb{R}^4$: it suffices to find its subperiods and to check that the associated equations, togerther with the Plücker relations, yield a zero-dimensional system.
One can even bound the diameter of the local rules by the length of the largest lift in $E$ of the subperiods.
Sharp bounds on this thickness however remain to be found.
In particular, when is it equal to one?
The proof of this theorem is postponed to the section \ref{sec:proof} and we shall first illustrate it on some examples.

%%%%%%%%%%%%%%%%%%%%%%%%%%%%%%%%%%%%%%
\subsection{First example: Ammann-Beenker tilings}
\label{sec:ammann_beenker}

Independently introduced by Ammann in the 1970s and Beenker in 1982 (\cite{GS,beenker}), the Ammann-Beenker tilings are the strongly planar rhombus tilings of codimension two whose slope is generated by the two vectors $(\cos(k\pi/4))_k$ and $(\sin(k\pi/4))_k$, $k=0,\ldots,3$.
The Grassmann coordinates of this slope are $(1,\sqrt{2},1,1,\sqrt{2},1)$.
There are four subperiods:
\begin{itemize}
\item the $123$-subperiod $\vec{p}_4:=\vec{e}_1-\vec{e}_3$ which corresponds to $G_{12}=G_{23}$;
\item the $124$-subperiod $\vec{p}_3:=\vec{e}_2+\vec{e}_4$ which corresponds to $G_{12}=G_{14}$;
\item the $134$-subperiod $\vec{p}_2:=\vec{e}_1+\vec{e}_3$ which corresponds to $G_{14}=G_{34}$;
\item the $234$-subperiod $\vec{p}_1:=\vec{e}_2-\vec{e}_4$ which corresponds to $G_{23}=G_{34}$.
\end{itemize}
Plugging this into the only one Plücker relation $G_{12}G_{34}=G_{13}G_{24}-G_{14}G_{23}$ with the normalization $G_{12}=1$ yields $G_{13}G_{24}=2$.
The system has thus dimension one and characterizes the family of planes
$$
E_0:=(0,0,0,0,1,0),\qquad
E_{t> 0}:=(1,t,1,1,2/t,1),\qquad
E_{\infty}:=(0,1,0,0,0,0).
$$
In particular, the slope of the Ammann-Beenker tilings is obtained for $t=\sqrt{2}$.
The subperiods lift in $E_{\sqrt{2}}$ onto the pairwise non-collinear vectors
$$
\vec{q}_1=\vec{p}_1+\sqrt{2}\vec{e}_1,
\qquad
\vec{q}_2=\vec{p}_2+\sqrt{2}\vec{e}_2,
\qquad
\vec{q}_3=\vec{p}_3+\sqrt{2}\vec{e}_3,
\qquad
\vec{q}_4=\vec{p}_4-\sqrt{2}\vec{e}_4.
$$
Theorem~\ref{th:planarity_codim2} ensures that the rhombus tilings with these subperiods are planar.\\

\begin{figure}[hbtp]
\includegraphics[width=0.31\textwidth]{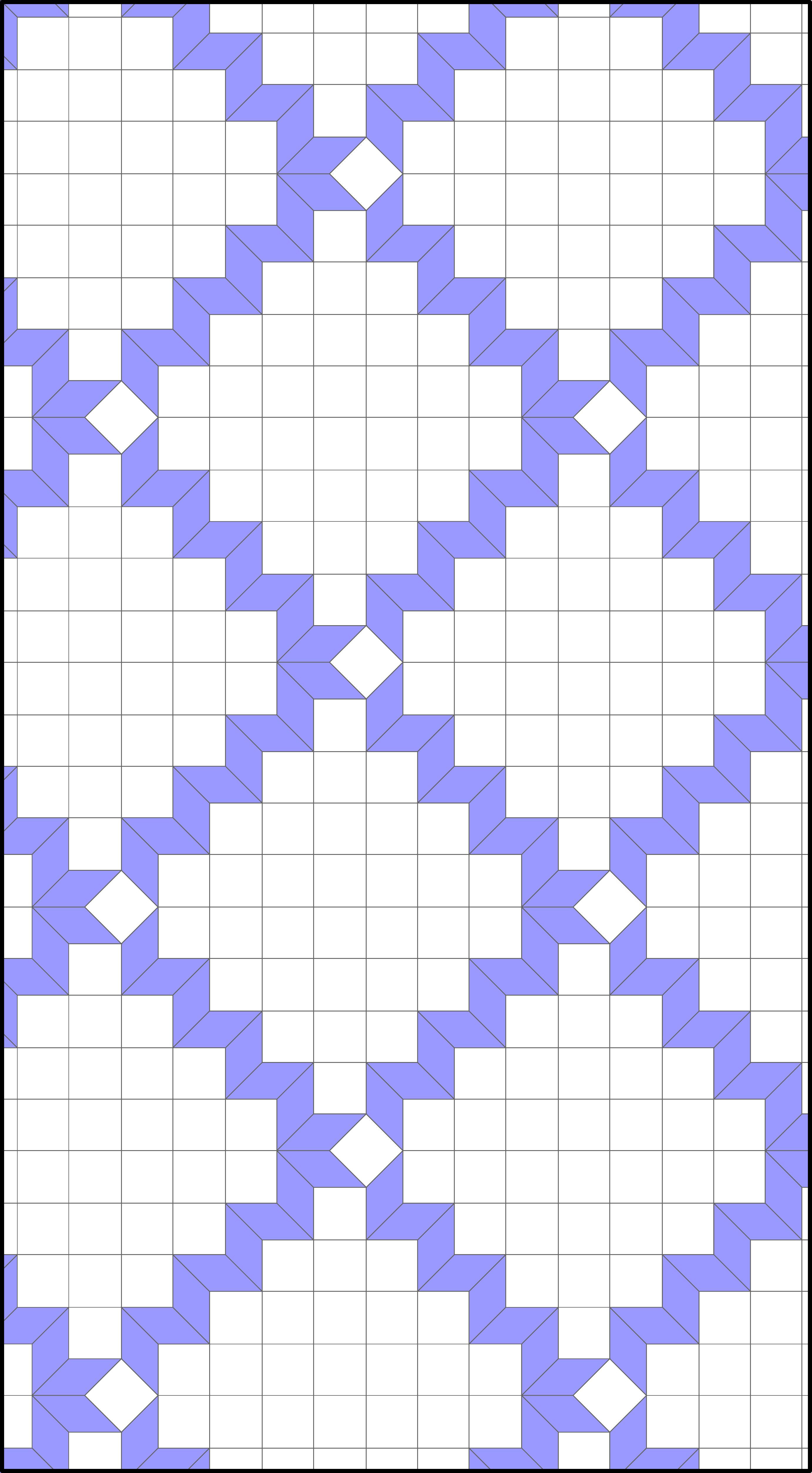}
\hfill
\includegraphics[width=0.31\textwidth]{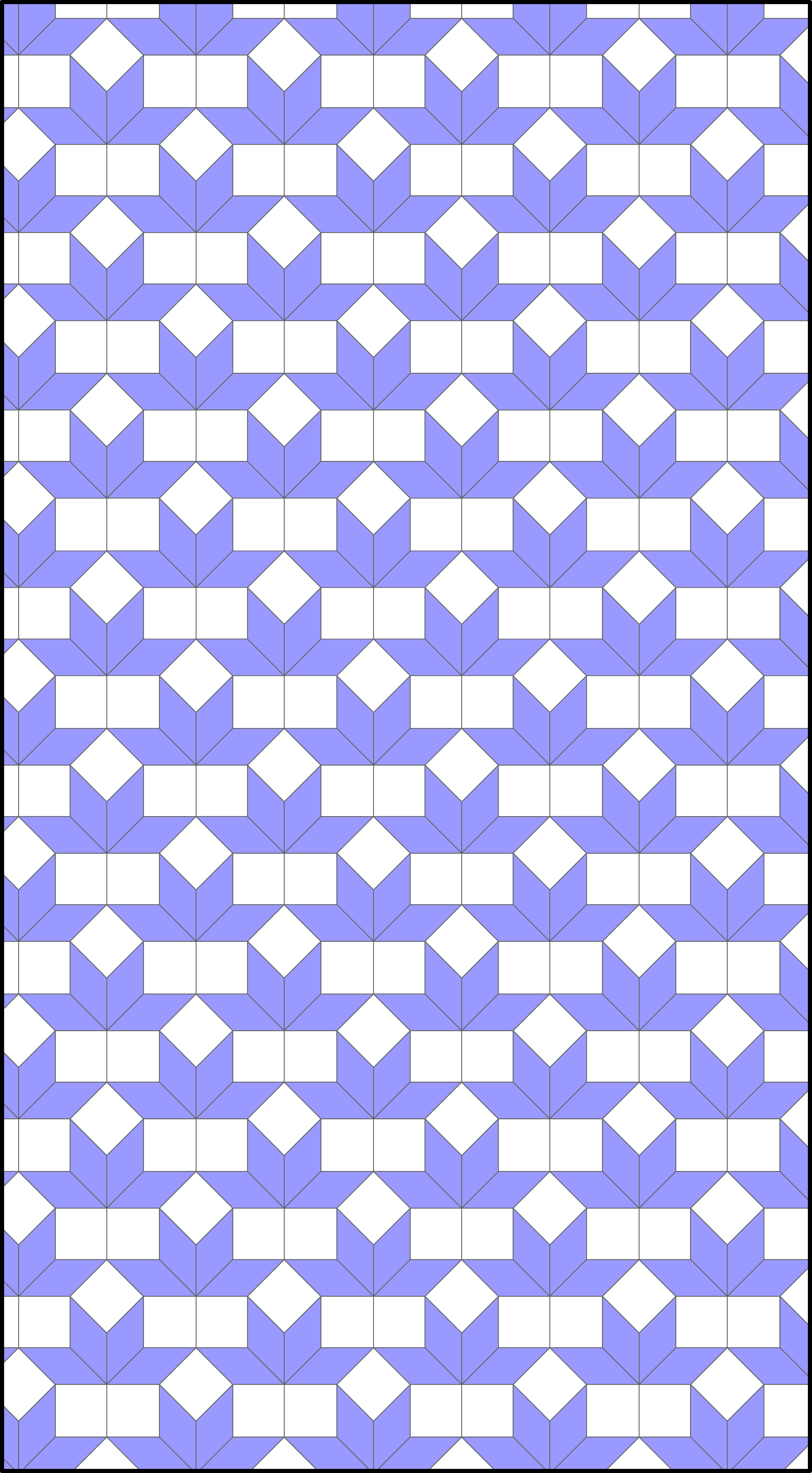}
\hfill
\includegraphics[width=0.31\textwidth]{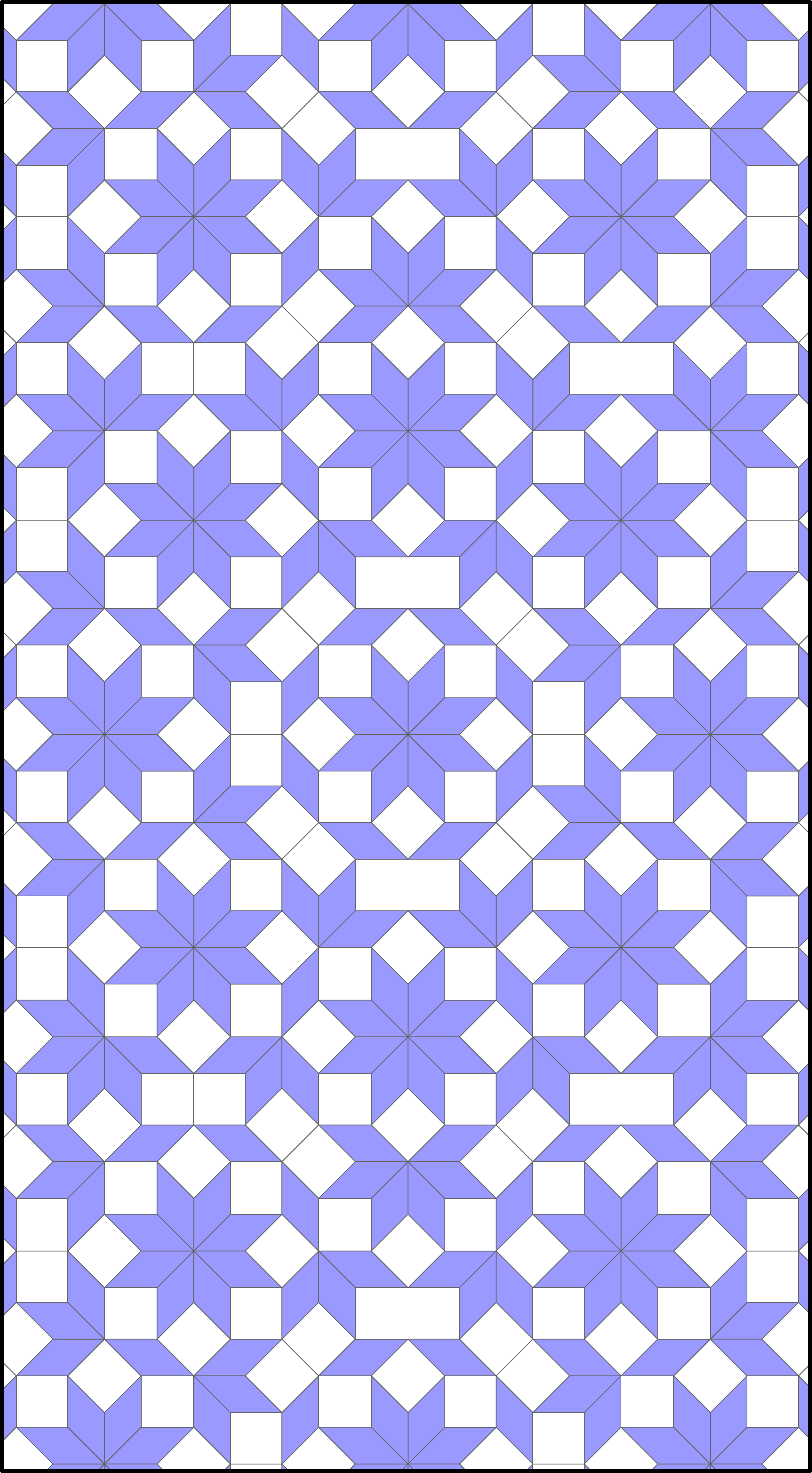}
\caption{
Strongly planar tilings with slope $E_{\frac{1}{4}}$, $E_1$ and $E_{\sqrt{2}}$ (from left to right).
They all have the same subperiods.
The rightmost is an Ammann-Beenker tiling and has the smallest proportion of square tiles.
}
\label{fig:ammann}
\end{figure}

Note that local rules can only enforce such subperiods for $t$ ranging in a closed interval (because the subperiods become arbitrarily large for large $t$), but one can color the local rules to enforce the whole family, see \cite{BF,katz}.
Note also that it would suffice to enforce $G_{13}=G_{24}$ (that is, according to Prop.~\ref{prop:frequences}, to enforce the square tiles $T_{13}$ and $T_{24}$ to appear with the same frequency) in order to characterize the slope of the Ammann-Beenker tilings.
This however cannot be done by local rules, as first pointed out by Burkov in \cite{burkov}: we need to use colored local rules, as first done by Ammann \cite{GS,A5}.
As an alternative, we can also obtain Ammann-Beenker tilings as the solution of an optimization problem.
Indeed, according to Prop.~\ref{prop:frequences}, the quantity $G_{13}+G_{24}=t+2/t$ is proportional to the frequency of the square tiles in $E_t$ and is minimal for $t=\sqrt{2}$.

%%%%%%%%%%%%%%%%%%%%%%%%%%%%%%%%%%%%%%
\subsection{Second example: a golden octagonal tiling}
\label{sec:golden_octagonal}

Let us now consider an example where subperiods characterize finitely many slopes.
Let $\varphi=\frac{1+\sqrt{5}}{2}$ be the golden ratio and $E$ the plane generated by
$$
(-1,0,\varphi,\varphi)
\qquad\textrm{and}\qquad
(0,1,\varphi,1).
$$
Its Grassmann coordinates are $E=(1,\varphi,1,\varphi,\varphi,1)$.
There are four subperiods:
\begin{itemize}
\item the $123$-subperiod $\vec{p}_4:=\vec{e}_1+\vec{e}_2$ which corresponds to $G_{13}=G_{23}$;
\item the $124$-subperiod $\vec{p}_3:=\vec{e}_2+\vec{e}_4$ which corresponds to $G_{12}=G_{14}$;
\item the $134$-subperiod $\vec{p}_2:=\vec{e}_1+\vec{e}_3$ which corresponds to $G_{14}=G_{34}$;
\item the $234$-subperiod $\vec{p}_1:=\vec{e}_3+\vec{e}_4$ which corresponds to $G_{23}=G_{24}$.
\end{itemize}
Plugging this into the Plücker relation $G_{12}G_{34}=G_{13}G_{24}-G_{14}G_{23}$ with the normalization $G_{12}=1$ yields $1=x^2-x$, where $x=G_{13}=G_{23}=G_{24}$.
Subperiods thus characterize $E$ and its algebraic conjugate\footnote{Only one really yields a tiling because their Grassmann coordinates have different sign.}, and Corollary~\ref{cor:local_rules_codim2} ensures that there are local rules.
One can check that the subperiods indeed lift in $E$ onto the pairwise non-collinear vectors
$$
\vec{q}_1=\vec{p}_1+(1-\varphi)\vec{e}_1,
\qquad
\vec{q}_2=\vec{p}_2+\varphi\vec{e}_2,
\qquad
\vec{q}_3=\vec{p}_3+\varphi\vec{e}_3,
\qquad
\vec{q}_4=\vec{p}_4+(1-\varphi)\vec{e}_4.
$$
These lifts have length at most $\sqrt{\varphi+3}$: this bounds the diameter of local rules.

\begin{figure}[hbtp]
\includegraphics[width=\textwidth]{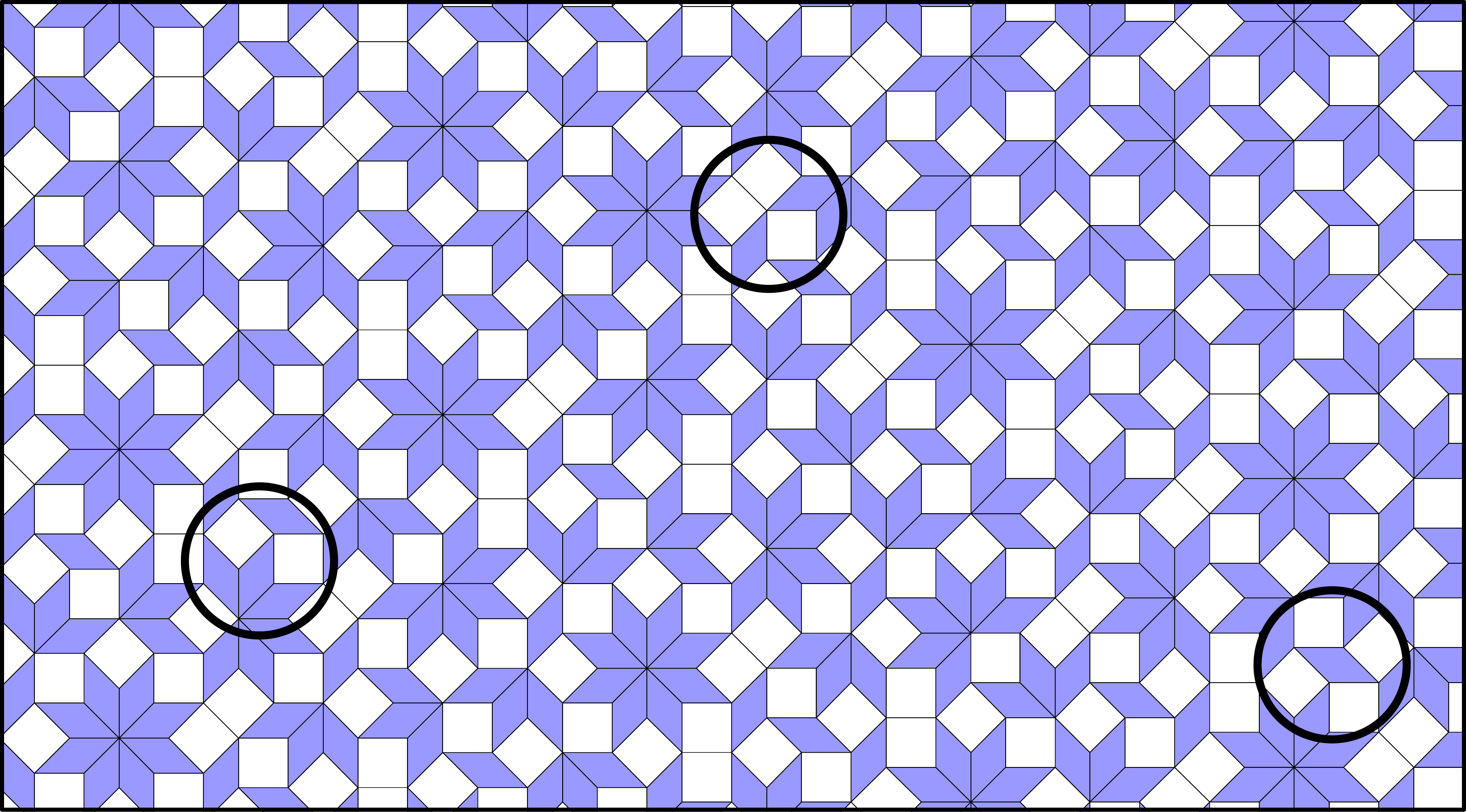}
\caption{
A strongly planar tiling with slope $(1,\varphi,1,\varphi,\varphi,1)$.
Since its sub\-pe\-riods characterize only finitely many slopes, it admits local rules of diameter less than the one of the depicted circles.}
\label{fig:prefere}
\end{figure}

\begin{figure}[hbtp]
\centering
\includegraphics[width=0.8\textwidth]{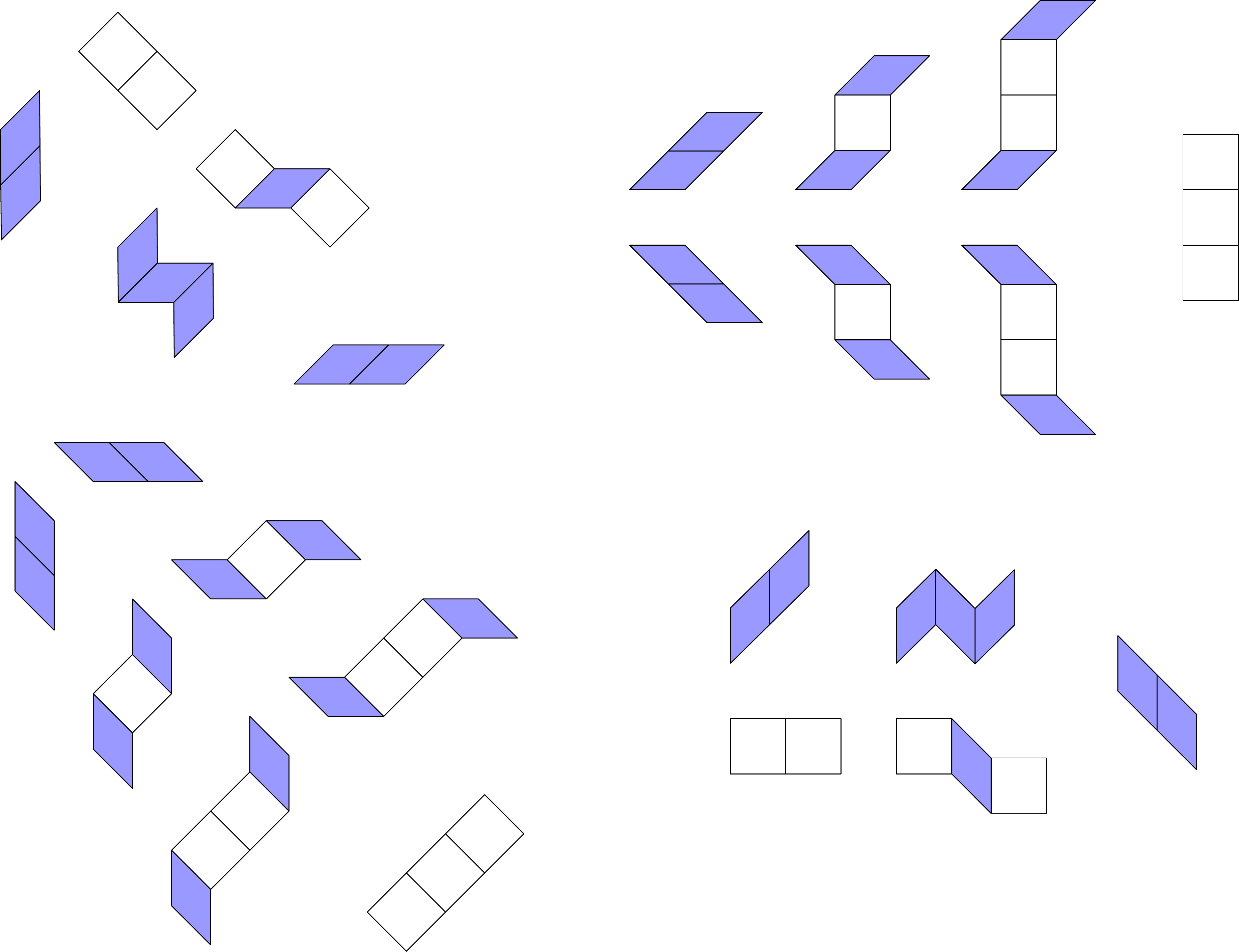}
\caption{
Any codimension $2$ tiling which avoids these $24$ forbidden patterns has the same subperiods as a strongly planar tiling with slope $(1,\varphi,1,\varphi,\varphi,1)$ (Fig.~\ref{fig:prefere}), hence is planar with the same slope.
Conversely, any strongly planar tiling with this slope avoids these $24$ forbidden patterns.}
\label{fig:prefere2}
\end{figure}

%%%%%%%%%%%%%%%%%%%%%%%%%%%%%%%%%%%%%%
\subsection{Proof of the main result}
\label{sec:proof}

The proof of Theorem~\ref{th:planarity_codim2} is organized in three lemmas.
The first lemma gives a condition on subperiods to ensure planarity:

\begin{lemma}\label{lem:planarity_codim2_CS}
If a rhombus tiling of codimension two has three subperiods, each in a shadow with only one period, which can be lifted in an irrational non-degenerated plane $E$ onto pairwise non-collinear vectors, then it is planar.
\end{lemma}

\begin{proof}
Let $\mathcal{T}$ be a codim. two tiling satisfying the condition of the Lemma.
Let $\vec{p}_1$, $\vec{p}_2$ and $\vec{p}_3$ denote the subperiods, each in a shadow with only one subperiod.
For $i\in\{1,2,3\}$, let $\vec{q}_i$ denotes the lift of $\vec{p_i}$ in $E$.\\

\noindent {\bf Space parametrization.}\\
The polynomial system defined by the Plücker relation and the linear relations associated with the subperiods of $E$ has at least two irrational solutions.
Indeed, if there are only finitely many solutions, then they are algebraic and each irrational one ({\em e.g}, $E$) yields by algebraic conjugation a different irrational solution.
Otherwise, that is, if there are infinitely many solutions, then they form a continuous curve in the set of planes of $\mathbb{R}^4$.
Since the set of planes which contain a rational line has measure zero (it is a countable union of dimension three subspaces of $\mathbb{R}^4$), this curve contains infinitely many irrational planes.
Let thus $E'$ be an irrational solution other than $E$.
We shall prove by contradiction that $E\cap E'=\{0\}$.
For $i\in\{1,2,3\}$, let $F_i$ be the plane generated by $\vec{q}_i$ and $\vec{e}_i$.
This defines three different rational planes.
Assume that $E\cap E'$ contains a line, necessarily irrational.
Hence $\dim(E+E')=3$.
This lines belongs to at most one of the $F_i$'s, say $F_3$, because if two rational planes intersect along a line, then it is a rational line.
And since $F_1$ and $F_2$ intersects $E$ and $E'$ by lines which generate it, one has $F_1+F_2\subset E+E'$.
We shall get the wanted contradiction by proving that $\dim(F_1+F_2)=4$.
With $\vec{p}_1=(a,b,c)$ and $\vec{p}_2=(d,e,f)$, one computes
$$
F_1=(a,b,c,0,0,0)
\qquad\textrm{and}\qquad
F_2=(-d,0,0,e,f,0).
$$
It is known (see, {\em e.g.}, \cite{HoP}, p.~304) that if the intersection of two $2$-planes of $\mathbb{R}^4$ with Grassmann coordinates $(A_{ij})$ and $(B_{ij})$ is not $\{0\}$, then
$$
A_{12}B_{34}-A_{13}B_{24}+A_{23}B_{14}+B_{12}A_{34}-B_{13}A_{24}+B_{23}A_{14}=0.
$$
In our case, $F_1\cap F_2\neq\{0\}$ would yield $bf-ce=0$.
But this is impossible because $E$, generated by $\vec{q}_1$ and $\vec{q}_2$, is non-degenerated and has the Grassmann coordinate $E_{34}=bf-ce$.
Hence $F_1\cap F_2=\{0\}$, that is, $\dim(F_1+F_2)=4$.\\

\noindent {\bf Lift parametrization.}\\
For $i\in\{1,2,3\}$, let $\vec{r}_i$ denotes the lift of $\vec{p_i}$ in $E'$.
Let $\pi$ denotes the projection parallel to $E'$ onto $E$.
Up to a permutation of the vectors of the standard basis of $\mathbb{R}^4$, one can assume that the angle between $\pi(\vec{e}_i)$ and $\pi(\vec{e}_j)$ has the same sign as the angle between $\vec{v}_i$ and $\vec{v}_j$ (the vectors defining the tile $T_{ij}$).
This way, if we let $\mathcal{S}$ be a lift of $\mathcal{T}$, then $\pi$ is a homeomorphism from $\mathcal{S}$ onto $E$.
On $E$, $\pi(\mathcal{S})$ is indeed nothing but the tiling $\mathcal{T}$ (up to a stretching of the edges of its tiles since $\pi(\vec{e}_i)$ and $\vec{v}_i$ can be different -- they are however never parallel because of the irrationality of $E'$).
There are thus two continuous functions $z_1$ and $z_2$ defined on $E$ such that $\mathcal{S}$ is the image of $E$ under
$$
\rho~:~\vec{x}\mapsto\vec{x}+z_1(\vec{x})\vec{r}_1+z_2(\vec{x})\vec{r}_2.
$$
We shall now show that $\rho$ stays at bounded distance from a plane.\\

\noindent {\bf From subperiods to bounded fluctuations.}\\
Let $\pi_i$ denote the projection onto the shadow which contains $\vec{p}_i$.
For any $\vec{x}\in E$, since the projection parallel to $E'$ is a homeomorphism from $\mathcal{S}$ onto $E$, the plane $\pi_i(\vec{x}+E')$ intersects the shadow $\pi_i(\mathcal{S})$ along a curve $\mathcal{C}_i(\vec{x})$ (see Fig.~\ref{fig:subperiod_constraint}).
One has
$$
\mathcal{C}_i(\vec{x})=\{\pi_i(\vec{x})+z_1(\vec{x}+\lambda\vec{q}_i)\pi_i(\vec{r}_1)+z_2(\vec{x}+\lambda\vec{q}_i)\pi_i(\vec{r_2})~|~\lambda\in\mathbb{R}\}.
$$
Since both $\pi_i(\mathcal{S})$ and $\pi_i(\vec{x}+E')$ are $\vec{p}_i$-periodic, so is $\mathcal{C}_i(\vec{x})$.
In particular, it stays at bounded distance from the line $\pi_i(\vec{x})+\mathbb{R}\vec{p}_i$.
Moreover, the bound can be chosen independently of $\vec{x}$ because $\mathcal{S}$ is Lipschitz.
For $i=1$, since $\pi_1(\vec{r}_1)=\vec{p}_1$, this ensures that $\lambda\mapsto z_2(\vec{x}+\lambda\vec{q}_1)$ is uniformly bounded.
In other words, $z_2$ has bounded fluctuations in the direction $\vec{q}_1$.
Similarly, for $i=2$, $\pi_2(\vec{r}_2)=\vec{p}_2$ yields that $z_1$ has bounded fluctuations in the direction $\vec{q}_2$.
For $i=3$, note that, up to a rescaling, one has $\vec{q}_3=\vec{q}_1+\alpha\vec{q}_2$ for some real $\alpha\neq 0$.
This allows to write
$$
\mathcal{C}_3(\vec{x})=\{\pi_3(\vec{x})+z_1(\vec{x}+\lambda\vec{q}_3)\pi_3(\vec{r}_3)+(z_2-\alpha z_1)(\vec{x}+\lambda\vec{q}_3)\pi_3(\vec{r_2})~|~\lambda\in\mathbb{R}\}.
$$
Then, with $\pi_3(\vec{r}_3)=\vec{p}_3$, the $\vec{p}_3$-periodicity of $\mathcal{C}_3(\vec{x})$ yields that $z_2-\alpha z_1$ has bounded fluctuations in the direction $\vec{q}_3$.\\

\begin{figure}[hbtp]
\centering
\includegraphics[width=0.9\textwidth]{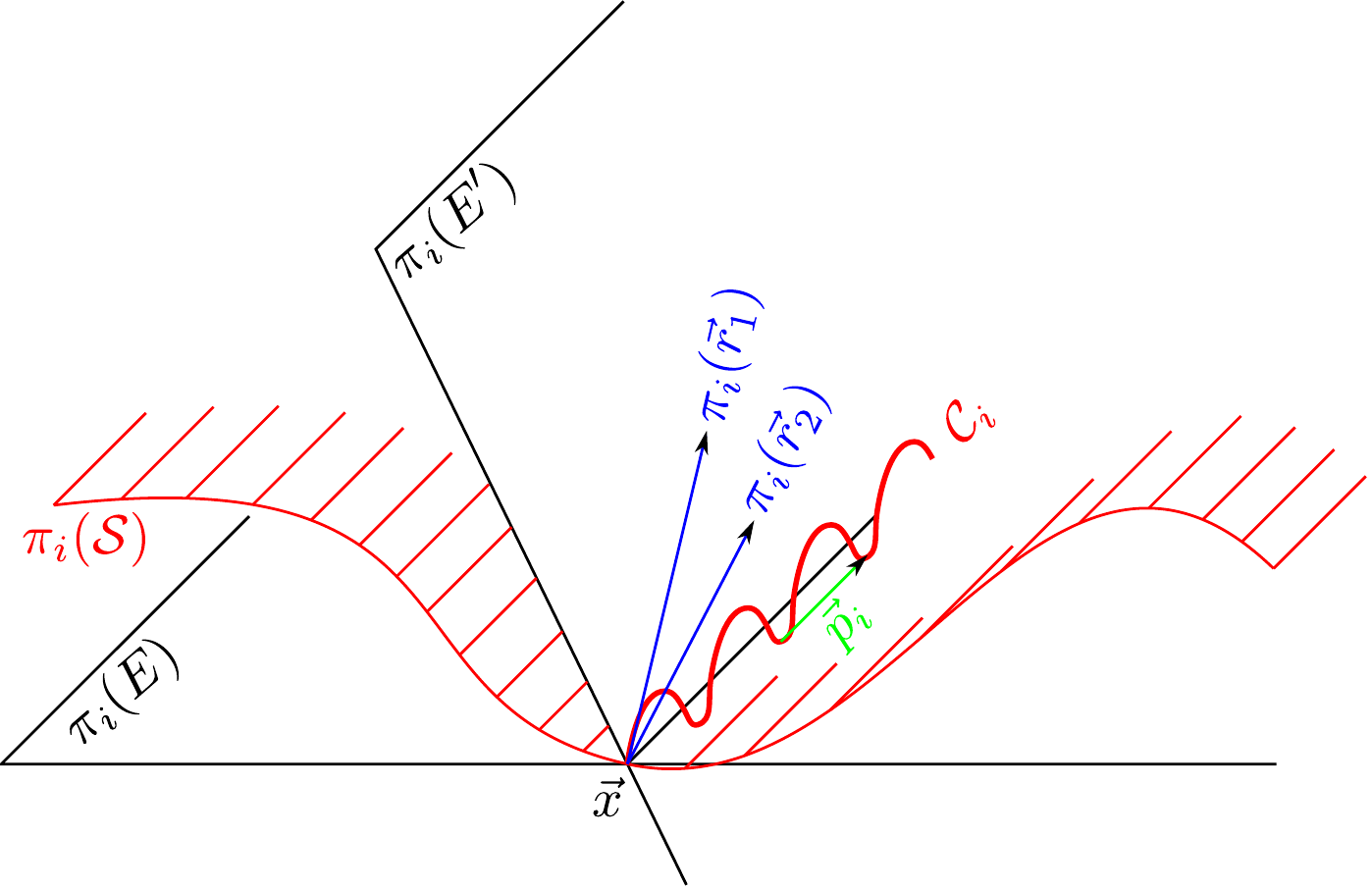}
\caption{
The projected lift $\pi_i(\mathcal{S})$ is $\vec{p}_i$-periodic: it intersects the plane $\pi_i(E')$, which contains $\vec{p}_i$, along a $\vec{p}_i$-periodic curve $\mathcal{C}_i$.
}
\label{fig:subperiod_constraint}
\end{figure}

\noindent {\bf From bounded fluctuations to functional equations.}\\
Since $\vec{q}_1$ and $\vec{q}_2$ form a basis of $E$, let $z_i(\lambda,\mu)$ stand for $z_i(\lambda\vec{q}_1+\mu\vec{q}_2)$, $i\in\{1,2\}$, and write $f\equiv g$ if the difference of two functions $f$ and $g$ is uniformly bounded.
The bounded fluctuations of $z_1$ and $z_2$ in the directions $\vec{q}_1$ and $\vec{q}_2$ yield the existence of real functions $f$ and $g$ such that $z_2(\lambda,\mu)\equiv f(\mu)$ and $z_1(\lambda,\mu)\equiv g(\lambda)$.
Further, since $\vec{q}_3=\vec{q}_1+\alpha\vec{q}_2$, the bounded fluctuations of $z_2-\alpha z_1$ in the direction $\vec{q}_3$ yield the existence of a real continuous function $h$ such that $(z_2-\alpha z_1)(\lambda,\mu)\equiv h(\lambda+\alpha\mu)$.
Thus
$$
f(\mu)-\alpha g(\lambda)\equiv h(\lambda+\alpha\mu).
$$

\noindent {\bf From functional equations to planarity.}\\
Fix $\lambda=0$ to get $f(\mu)\equiv h(\alpha\mu)$.
Fix $\mu=0$ to get $-\alpha g(\lambda)\equiv h(\lambda)$.
Hence
$$
h(\alpha\mu)+h(\lambda)\equiv h(\lambda+\alpha\mu).
$$
Since $\alpha\neq 0$, one can replace $\alpha\mu$ by $\mu$, getting the functional equation
$$
h(\mu)+h(\lambda)\equiv h(\lambda+\mu).
$$
This easily yields the linearity of $h$ (up to bounded fluctuations), thus the linearity of $f$, $g$, $z_1$, $z_2$ and, finally, $\rho$.
The thickness is moreover uniformly bounded because the lifts are all Lipschitz surfaces with a common constant.
This completes the proof.
\end{proof}

\noindent The second lemma shows that the condition on subperiods is actually necessary:

\begin{lemma}\label{lem:planarity_codim2_CN}
If the subperiods of a codimension two rhombus tiling enforce irrational planarity then three of them, each in a shadow with only one subperiod, can be lifted in an irrational non-degenerated plane onto pairwise non-collinear vectors.
\end{lemma}

\begin{proof}
Let $\mathcal{T}$ be a planar codim. $2$ tiling with an irrational non-degenerated slope $E$ whose sub\-pe\-riods $\vec{p}_1,\ldots,\vec{p}_k$ enforce planarity.
As in Lem.~\ref{lem:planarity_codim2_CS}, there is a plane $E'$ with at least the subperiods of $E$ such that $E\cap E'=\{0\}$.
Let $\pi_i$ denote the projection onto the shadow which contains $\vec{p}_i$.
Let $\vec{q}_i$ and $\vec{r}_i$ denote the lift of $\vec{p_i}$ respectively in $E$ and $E'$.
The proof shall be by contradiction.
Let us separate two cases.\\

\noindent {\bf Case 1: The subperiods, once lifted in $E$, belong to at most two lines.}\\
Assume that there are exactly two such lines, say $\mathbb{R}\vec{q}_1$ and $\mathbb{R}\vec{q}_2$ (this is all the more true if there are only one line).
For any two real functions of a real variable $f$ and $g$, define
$$
\mathcal{S}_{f,g}:=\{\lambda\vec{q}_1+\mu\vec{q}_2+f(\lambda)\vec{r}_1+g(\mu)\vec{r}_2~|~\lambda,\mu\in\mathbb{R}\}.
$$
For $i$ such that the lift of $\vec{p}_i$ belong to $\mathbb{R}\vec{q}_1$, say $\vec{p}_i=\alpha_i\vec{q}_1$, one has
$$
\pi_i(\mathcal{S}_{f,g})=\{(\lambda+f(\lambda))\frac{1}{\alpha_i}\vec{p}_i+\mu\pi_i(\vec{q}_2)+g(\mu)\pi_i(\vec{r}_2)~|~\lambda,\mu\in\mathbb{R}\}.
$$
It follows that $\vec{p}_i$ is a period of $\pi_i(\mathcal{S}_{f,g})$ as soon as $\{\lambda+f(\lambda)~|~\lambda\in\mathbb{R}\}$ is stable under the translation $x\to x+\alpha_i$.
Similarly, for $i$ such that the lift of $\vec{p}_i$ belong to $\mathbb{R}\vec{q}_2$, say $\vec{p}_i=\beta_i\vec{q}_2$, $\vec{p}_i$ is a period of $\pi_i(\mathcal{S}_{f,g})$ as soon as $\{\mu+g(\mu)~|~\mu\in\mathbb{R}\}$ is stable under the translation $x\to x+\beta_i$.
For such functions $f$ and $g$, consider a tiling whose lift lies in $\mathcal{S}_{f,g}+[0,1]^n$ (that is, an approximation of $\mathcal{S}_{f,g}$).
It has the same subperiods as $\mathcal{T}$.
But it is not necessarily planar: take, for example, $f(x)=g(x)=x^3$.
This yields the wanted contradiction.\\

\noindent {\bf Case 2: There are at most two shadows with only one subperiod.}\\
Assume that there are exactly two such shadows, say those with subperiods $\vec{p}_1$ and $\vec{p}_2$ (this is all the more true if there are less such shadows)
For $i\in\{1,2\}$, let $\vec{q}_i$ and $\vec{r}_i$ be the lifts of $\vec{p}_i$, respectively in $E$ and $E'$.
We define $\mathcal{S}_{f,g}$ as above.
For $i\in\{1,2\}$, $\pi_i(\vec{q}_i)=\pi_i(\vec{r}_i)=\vec{p}_i$, so that $\vec{p}_i$ is a period of $\pi_i(\mathcal{S}_{f,g})$ as soon as $\{\lambda+f(\lambda)~|~\lambda\in\mathbb{R}\}$ is stable under the translation $x\to x+1$.
For $i\notin\{1,2\}$, $\pi_i(E)$ has at least one subperiod by definition of the $\pi_i$'s, hence two because of our initial hypothesis.
This is thus a rational plane of $\mathbb{R}^3$, hence equal to its algebraic conjugate $\pi_i(E')$.
It follows that $\pi_i(\mathcal{S}_{f,g})=\pi_i(E)=\pi_i(E')$.
In particular, $\pi_i(\mathcal{S}_{f,g})$ is $\vec{p}_i$-periodic.
So, again, we can choose $f$ and $g$ to obtain a non planar tiling which has the same subperiods as $\mathcal{T}$.
This yields the wanted contradiction.
\end{proof}

The last lemma shows that the condition on subperiods is satisfied in particular when superiods characterize only finitely many slopes (that is a necessary condition to have local rules with our method):

\begin{lemma}\label{lem:planarity_codim2_CP}
If the subperiods of a codimension two rhombus tiling characterize finitely many slopes, then they enforce irrational planarity.
\end{lemma}

\begin{proof}
If the subperiods do not enforce irrational planarity, then one can take $f(\lambda)=\alpha\lambda$ and $g(\mu)=\beta\mu$ for any $\alpha\neq-1$ and $\beta\neq -1$ in the proof of the previous lemma: this yields infinitely many slopes with these subperiods.
\end{proof}

%%%%%%%%%%%%%%%%%%%%%%%%%%%%%%%%%%%%%%
\section{Higher codimension}
\label{sec:higher_codim}

%%%%%%%%%%%%%%%%%%%%%%%%%%%%%%%%%%%%%%
\subsection{A partial result and a conjecture}
\label{sec:statement_higher_codim}

In codimension two (that is, for tilings whose lift lives in $\mathbb{R}^4$), Theorem~\ref{th:planarity_codim2} provides a necessary and sufficient condition on the subperiods of a tiling to ensure that it is planar.
The codimension two case can then be helpful to solve higer codimension cases.
Indeed, consider the projections of a tiling onto the space generated by four basis vector (those are a kind of generalization of the shadows, Def~\ref{def:subperiod}).
This yields codimension two tilings to which Theorem~\ref{th:planarity_codim2} can be applied.
We can then use the (eventual) planarity of these projections to (eventually) get the planarity of the original tiling.
We shall see successful cases in the following sections, namely the famous Penrose tilings (actually, a slightly generalized version) and a codimension four tiling based on cubic irrationality (whose main interest, beyond illustrating the method, is to show that cubic irrationality can be already obtained in codimension four).\\

However, we think that there are tilings whose subperiods enforce planarity, although no projection on four basis vectors does have subperiods which enforce its planarity.
That is, the above method is not expected to always work.
Moreover, this do not provide a full characterization of planarity in higher codimension.
Nevertheless, we conjecture that Corollary~\ref{cor:local_rules_codim2} naturally extends:
\begin{conjecture}
If there are only finitely many slopes with the same subperiods as a given slope, then this slope admits local rules.
\end{conjecture}
In other words, we conjecture that if subperiods yield enough constraints on planar tilings to enforce their slope, then they {\em a fortioti} yield enough constraints on tilings to enforce their planarity.

%%%%%%%%%%%%%%%%%%%%%%%%%%%%%%%%%%%%%%
\subsection{First example: generalized Penrose tilings}
\label{sec:generalized_penrose}

Discovered by Penrose in the 70's \cite{penrose2}, the Penrose tilings appear in a number of versions (see, {\em e.g.}, \cite{GS}).
Thoses with rhombus tiles have been shown by de Bruijn \cite{debruijn} to be strongly planar with a lift in $\mathbb{R}^5$ whose slope is generated by the two vectors $(\cos(2k\pi/5))_k$ and $(\sin(2k\pi/5))_k$, $k=0,\ldots,4$.
This slope has Grassmann coordinates $(\varphi,1,-1,-\varphi,\varphi,1,-1,\varphi,1,\varphi)$, where $\varphi$ is the golden ratio, and can also be generated by
$$
\vec{u}:=(\varphi,0,-\varphi,-1,1)
\qquad\textrm{and}\qquad
\vec{v}:=(-1,1,\varphi,0,-\varphi).
$$
We here consider so-called generalized Penrose tilings, introduced in \cite{KP}, which are the strongly planar rhombus tilings whose slope is parallel to the one of Penrose tilings (recall that the slope is an affine plane).
They have ten subperiods (one in each shadow), associated with the equations
$$
G_{12}=G_{23}=G_{34}=G_{45}=-G_{15}
\quad\textrm{ and }\quad
G_{13}=G_{35}=-G_{25}=G_{24}=-G_{14}.
$$
Let us normalize to $G_{12}=1$ and write $G_{13}=x$.
There are five Plücker relations, which all reduce to the unique equation $x^2=x+1$, so that $x$ is equal to the golden ratio or its algebraic conjugate.
The subperiods thus characterize finitely many slopes: it suffices to show that they also enforce planarity to prove that generalized Penrose tilings admit local rules.
For that, project the slope onto the first four basis vectors.
It yields a slope $(\varphi,1,-1,\varphi,1,\varphi)$ which has four subperiods associated with the equations
$$
G_{12}=G_{23}=G_{34}
\qquad\textrm{and}\qquad
G_{13}=-G_{14}=G_{24}.
$$
We can thus apply Theorem~\ref{th:planarity_codim2}: this projection stays at bounded distance from the plane $(\varphi,1,-1,\varphi,1,\varphi)$.
Consider now the cartesian product of this plane with the line generated by the fifth basis vector: it is a three-dimensional vectorial space from which the tiling stays at bounded distance.
The same holds (by circular permutation of the indices) for the other projections on four of the five basis vectors, so that the tiling stays at bounded distance from the intersection of five three-dimensional vector spaces.
The two-dimensionality of this intersection shall yields the planarity of the tilings.
Consider a point $\vec{x}=(x_1,x_2,x_3,x_4,x_5)$ in this intersection.
There are real numbers $\lambda_i$ and $\mu_i$ such that $\pi_i(\vec{x})=\lambda_i\pi_i(\vec{u})+\mu_i\pi_i(\vec{v})$, where $\pi_i$ denotes the projection on the space generated by all the basis vectors but $\vec{e}_i$.
One checks that the $\lambda_i$'s are necessarily all equal to $-x_4$, and that the $\mu_i$'s are necessarily all equal to $x_2$.
This yields the wanted two-dimensionality of the intersection.
In conclusion, as conjectured in \cite{KP} and later proven in \cite{socolar}, the generalized Penrose tilings admit local rules.
Namely, the local rules which enforce the subperiods of the generalized Penrose tilings (Prop.~\ref{prop:subperiod_rules}).
One also has a bound on the diameter of the local rules, namely the largest subperiod lift (in the Penrose slope).
A computation yields the bound $\sqrt{2+2\varphi^2}\simeq 2.69$.

\begin{figure}[hbtp]
\includegraphics[width=\textwidth]{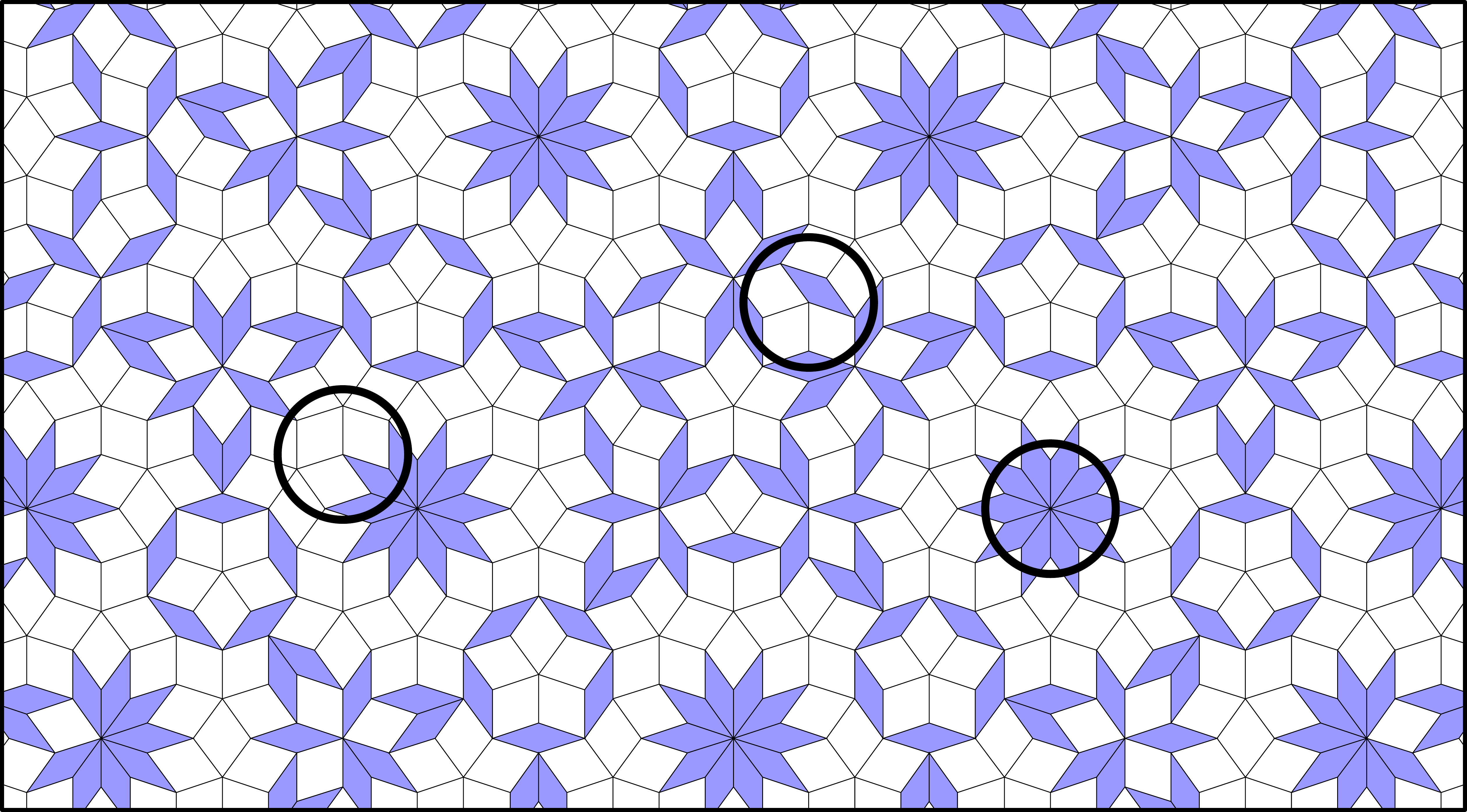}
\caption{A generalized Penrose tilings (compare with Fig.~\ref{fig:subperiods}), with circles bounding the diameter of the local rules it admits.}
\label{fig:penrose}
\end{figure}

\begin{figure}[hbtp]
\includegraphics[width=\textwidth]{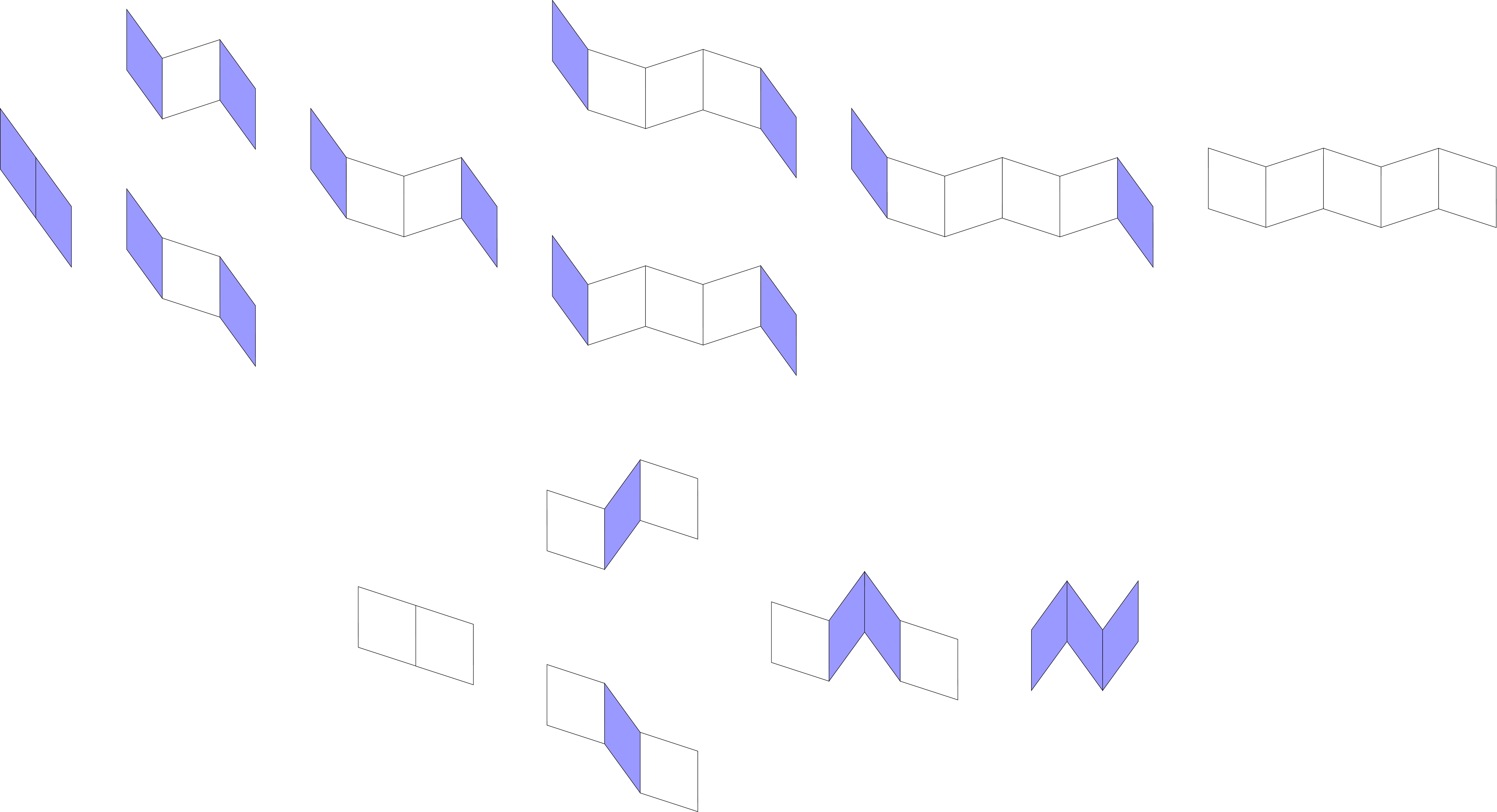}
\caption{
A set of forbidden patterns (depicted up to rotation by an angle multiple of $\frac{2\pi}{5}$ and mirror symmetry) which enforce the subperiods of generalized Penrose tilings, hence their slope (compare with Fig.~\ref{fig:forcing_subperiods}).}
\label{fig:penrose2}
\end{figure}

%%%%%%%%%%%%%%%%%%%%%%%%%%%%%%%%%%%%%%
\subsection{Second example: a cubic dodecagonal tiling}
\label{sec:cubic_dodecagonal}

Let us consider an example in $\mathbb{R}^6$, namely the codimension four planes satisfying
$$
G_{12}=G_{23}=G_{34}=G_{45}=G_{56},
$$
$$
G_{35}=G_{13}=G_{16}=G_{46}=G_{24},
$$
$$
G_{14}=G_{15}=G_{25}=G_{26}=G_{36}.
$$
According to Propositions~\ref{prop:subperiod_rules} and \ref{prop:subperiod_grassmann}, these relations on Grassmann coordinates can be enforced by local rules.
One checks that, together with the $\binom{6}{4}$ Plücker relations, this form a zero-dimensional system with three real solutions:
$$
(1,a,b,b,a,1,a,b,b,1,a,b,1,a,1),
$$
where $a^3=a^2+2a-1$ and $b=a^2-1$.
A basis is given by
$$
(-1,0,1,a,b,b)
\quad\textrm{and}\quad
(0,1,a,b,b,a).
$$
It remains to show that subperiods also enforce planarity.
Fix a tiling $\mathcal{T}$ which has the above subperiods.
We shall consider two of its projections.
First, project orthogonally onto the space generated by $\vec{e}_1$, $\vec{e}_2$, $\vec{e}_3$ and $\vec{e}_5$.
The subperiods yield
$$
G_{12}=G_{23}
\qquad\textrm{and}\qquad
G_{35}=G_{13}
\qquad\textrm{and}\qquad
G_{15}=G_{25}.
$$
We can thus apply Theorem~\ref{th:planarity_codim2}: this projection of $\mathcal{T}$ stays at bounded distance from a plane, which can only be the projection of a solution of the whole system, that is, $(1,a,b,1,b,a)$, which is generated, {\em e.g.}, by $(-1,0,1,b)$ and $(0,1,a,b)$.
Second, project orthogonally onto the space generated by $\vec{e}_1$, $\vec{e}_4$, $\vec{e}_5$ and $\vec{e}_6$.
The subperiods yield
$$
G_{45}=G_{56}
\qquad\textrm{and}\qquad
G_{16}=G_{46}
\qquad\textrm{and}\qquad
G_{14}=G_{15}.
$$
We can thus apply Theorem~\ref{th:planarity_codim2}: this projection of $\mathcal{T}$ stays at bounded distance from a plane, which can only be the projection of a solution of the whole system, that is, $(b,b,a,1,a,1)$, which is generated, {\em e.g.}, by $(-b,0,1,a)$ and $(0,b,b,a)$.
Now, consider the vectorial space $V\subset\mathbb{R}^6$ which projects onto the two above slopes.
The lift of $\mathcal{T}$ thus stays at bounded distance from $V$.
The planarity shall follow once we prove that $V$ has dimension at most two.
Let $(x_1,x_2,x_3,x_4,x_5,x_6)\in V$.
There are numbers $\lambda_1$, $\mu_1$, $\lambda_2$ and $\mu_2$ such that
$$
\begin{array}{ccccc}
x_1 &=& -\lambda_1 &=& -b\lambda_2\\
x_2 &=& \mu_1 &=& \\
x_3 &=& \lambda_1+a\mu_1 &=&\\
x_4 &=& &=& b\mu_2\\
x_5 &=& b\lambda_1+b\mu_1 &=& \lambda_2+b\mu_2\\
x_6 &=& &=& a\lambda_2+a\mu_2
\end{array}
$$
One easily checks that these equations yield that $x_3$, $x_4$, $x_5$ and $x_6$ are completly determined by $x_1$ and $x_2$.
This shows that $V$ has dimension at most two, whence the planarity of $\mathcal{T}$.\\

In conclusion, the slope $(1,a,b,b,a,1,a,b,b,1,a,b,1,a,1)$, where $a$ is a root of $X^3-X^2-2X+1$ and $b=a^2-1$, does admit local rules.
One also has a bound on the diameter of the local rules, namely the largest subperiod lift (in the above slope): a computation yields the upper bound $2,821$.

\begin{figure}[hbtp]
\includegraphics[width=\textwidth]{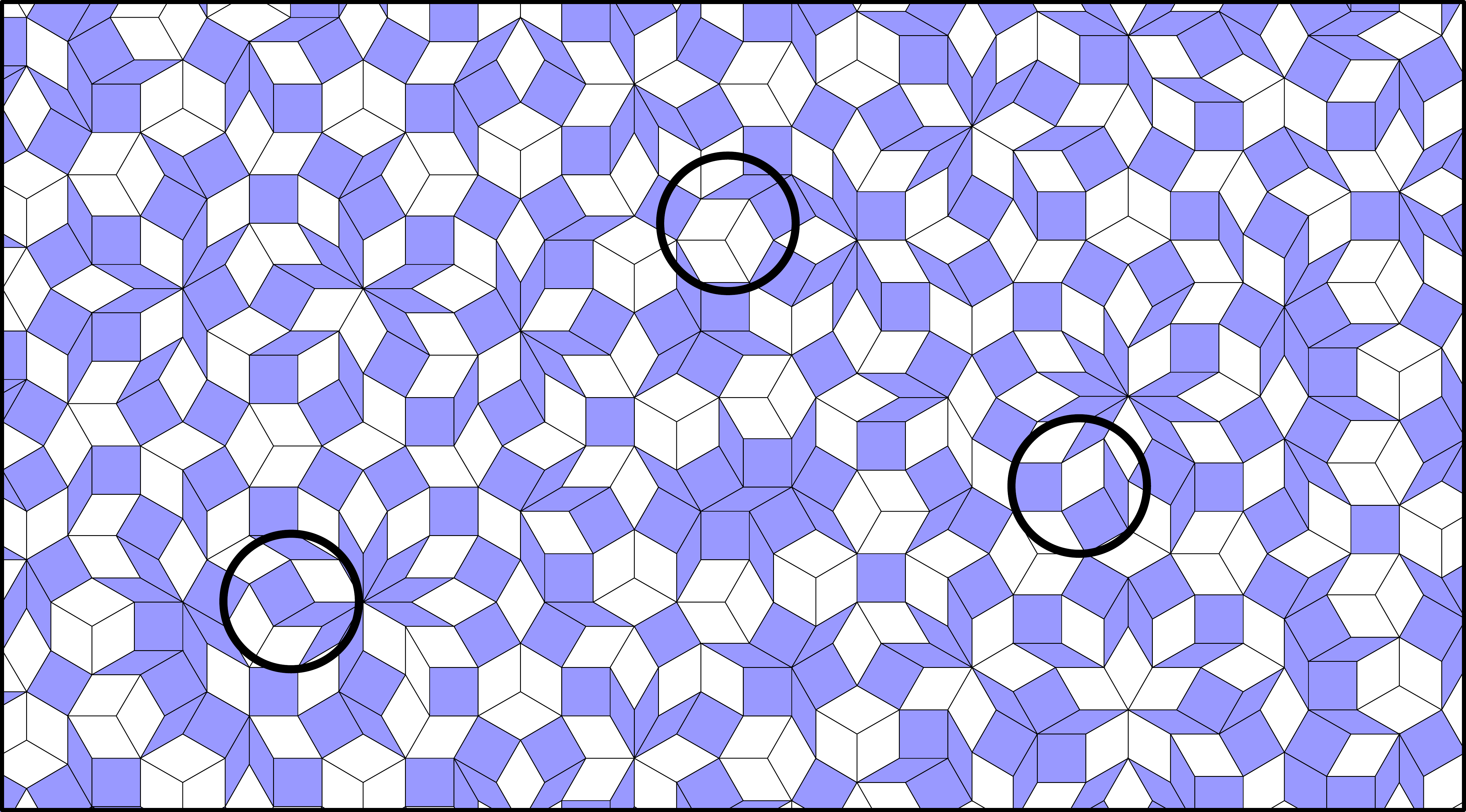}
\caption{A codimension four strongly planar tiling whose slope is based on cubic irrationalities, with circles bounding the diameter of the local rules it admits.}
\label{fig:cubic_dodecagonal}
\end{figure}

%%%%%%%%%%%%%%%%%%%%%%%%%%%%%%%%%%%%%%
\subsection{Tilings with $n$-fold rotational symmetry}
\label{sec:n_fold}

\begin{definition}
A rhombus tiling is said do be {\em $n$-fold} if it is strongly planar with a slope parallel to the plane generated by $(\cos(2k\pi/n))_k$ and $(\sin(2k\pi/n))_k$, where $k$ range from $0$ to either $n-1$ if $n$ is odd, or from $0$ to $n/2-1$ if $n$ is even.
\end{definition}

\begin{figure}[hbtp]
\includegraphics[width=0.46\textwidth]{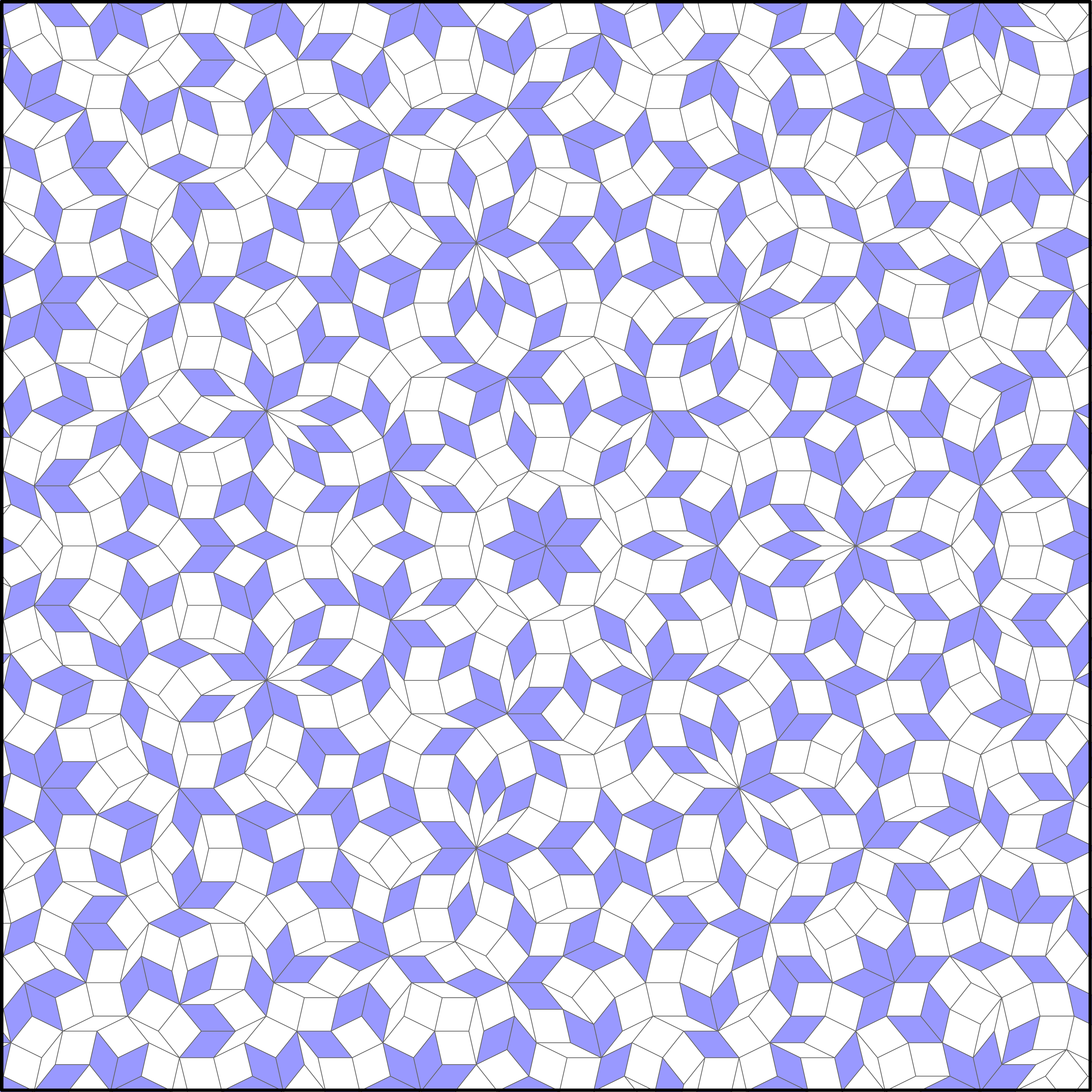}
\hfill
\includegraphics[width=0.46\textwidth]{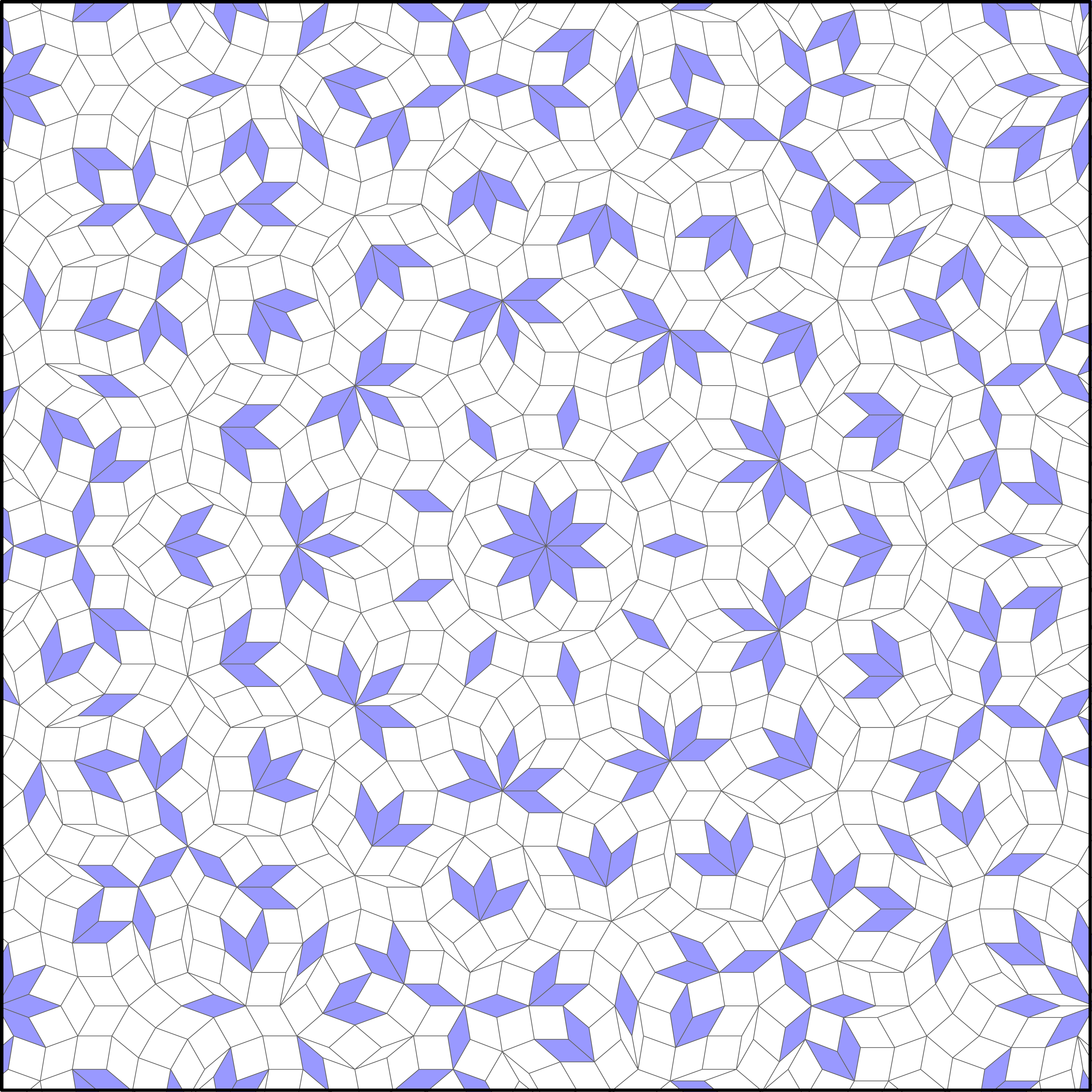}
\caption{A $7$-fold (left) and a $9$-fold (right) strongly planar tilings.}
\label{fig:n_fold}
\end{figure}

The name comes from the fact that these tilings contain arbitrarily large balls with a $n$-fold rotational symmetry (one speaks about {\em local} $n$-fold symmetry).
Fig.~\ref{fig:n_fold} illustrates the cases $n=7$ and $n=9$.
Let us stress that $n$-fold tilings lift in $\mathbb{R}^n$ for odd $n$, but in $\mathbb{R}^{n/2}$ for even $n$: this is because adopting the same definition for both cases would yield $n/2$ pairs of collinear vectors for even $n$.
We already met $n$-fold tilings in this paper: the Ammann-Beenker tilings (Fig.~\ref{fig:ammann}) are indeed $8$-fold and the generalized Penrose tilings (Fig.~\ref{fig:penrose}) are $5$-fold.\\

\noindent The Grassmann coordinates of a $n$-fold tiling are
$$
G_{ij}=\sin\left(\frac{2(j-i)\pi}{n}\right).
$$
It shall be convenient to set $G_{i,j+n}=G_{i,j}$ for $n$ odd, $G_{i,j+n/2}=-G_{i,j}$ for $n$ even, and $G_{ji}=-G_{ij}$ for any $n$.
For $i<j$, there is a subperiod associated with
$$
G_{ij}=G_{j,2j-i}.
$$
There are no other subperiod, except for $n=12p$, where the rationality of $\sin(\pi/6)$ yields for each $i$ two subperiods associated with
$$
G_{i,i+3p}=2G_{i,i+p}
\quad\textrm{and}\quad
G_{i,i+3p}=2G_{i,i+5p}.
$$
We say that a set of Grassmann coordinates are {\em free} if each of them can be chosen independently withou violating the Plücker relations.
We shall use:

\begin{lemma}\label{lem:G12G13}
The $G_{ij}$'s with $|j-i|\leq 2$ are free and determine all the other ones.
\end{lemma}

\begin{proof}
We first prove by induction on $\delta$ that these Grassmann coordinates determine those with $|j-i|\leq\delta$.
There is nothing to prove for $\delta=1$ and $\delta=2$.
Fix $\delta\geq 3$ and assume that any Grassmann coordinate $G_{ij}$ with $|j-i|< \delta$ is characterized.
Then, for $|j-i|=\delta$, the Plücker relation
$$
G_{i,i+1}G_{j-1,j}=G_{i,j-1}G_{i+1,j}-G_{ij}G_{i+1,j-1}
$$
shows that $G_{ij}$ depends only on coordinates $G_{kl}$ with $|k-l|<|j-i|=\delta$.
The claim follows by induction.
Now, since there are as many Grassmann coordinates with $|j-i|\leq 2$ as coordinates in two vectors which generate a plane (that is, twice the dimension of the space), these Grassmann coordinates are free.
\end{proof}

We first show that, except when $n$ is a multiple of $4$, the only planar rhombus tilings with the same subperiods as the $n$-fold tilings are the $n$-fold tilings:

\begin{proposition}\label{prop:n_fold_dim0}
If $4$ does not divide $n$, then the Plücker relations and thoses associated with the subperiods of a $n$-fold tiling form a zero-dimensional system.
\end{proposition}

\begin{proof}
Let $m:=n$ if $n$ is odd, or $m:=n/2$ if $n$ is even.
Subperiods enforce
$$
G_{12}=G_{23}=G_{34}=\ldots=G_{m-1,m}=G_{m,m+1}.
$$
Since $4$ does not divide $n$, $m$ is odd, and subperiods enforce
$$
G_{13}=G_{35}=\ldots=G_{m-2,m}=G_{m,m+2}=G_{24}=G_{46}=\ldots=G_{m-1,m+1}.
$$
The Plücker relation
$$
G_{1,i}G_{i+1,i+2}=G_{1,i+1}G_{i,i+2}-G_{1,i+2}G_{i,i+1}
$$
can then be rewritten
$$
G_{1,i}G_{12}=G_{1,i+1}G_{13}-G_{1,i+2}G_{12}.
$$
With $X:=G_{13}/(2G_{12})$ and $U_i:=G_{1,i+2}/G_{12}$, this yields the recurrence relation
$$
U_0=1,\quad
U_1=2X,\quad
U_i=2X U_{i-1}-U_{i-2},
$$
which is exactly the one defining Chebyshev polynomials of the second kind.
Thus $X$ is one of the finitely many solutions of $U_{m-2}=G_{1,m}/G_{12}$.
%One computes $X=\cos(2\pi/n)$.
The zero-dimensionality follows from Lemma~\ref{lem:G12G13}.
\end{proof}

In contrast, when $n$ is a multiple of $4$, there is a one-parameter family of planar rhombus tilings with the same subperiods as the $n$-fold tilings:

\begin{proposition}\label{prop:n_fold_dim1}
If $4$ divides $n$, then the Plücker relations and thoses associated with the subperiods of a $n$-fold tiling form a one-dimensional system.
\end{proposition}

\begin{proof}
Let $m:=n$ if $n$ is odd, or $m:=n/2$ if $n$ is even.
Subperiods enforce
$$
G_{12}=G_{23}=G_{34}=\ldots=G_{m-1,m}=G_{m,m+1}.
$$
Since $4$ divides $n$, $m$ is even, and subperiods now only enforce
$$
G_{13}=G_{35}=\ldots=G_{m-1,m+1}
\quad\textrm{and}\quad
G_{24}=G_{46}=\ldots=G_{m,m+2}.
$$
With $X:=G_{13}/(2G_{12})$, $Y:=G_{24}/(2G_{12})$ and $U_i:=G_{1,i+2}/G_{12}$, the relation
$$
G_{1,i}G_{i+1,i+2}=G_{1,i+1}G_{i,i+2}-G_{1,i+2}G_{i,i+1}
$$
now yields the recurrence relation
$$
U_0=1,\quad
U_1=2X,\quad
U_{2i}=2Y U_{2i-1}-U_{2i-2},\quad
U_{2i+1}=2X U_{2i}-U_{2i-1}.
$$
Hence $U_i$ is obtained from the $i$-th Chebyshev polynomial of the second kind by replacing $X^{2k+1}$ by $X^{k+1}Y^k$ and $X^{2k}$ by $X^kY^k$.
In particular, since $m-2$ is even, $U_{m-2}$ contains only powers of $XY$, so that $XY$ is the square of a solution of $U_{m-2}=G_{1,m}/G_{12}=1$.
%One computes $XY=\cos^2(2\pi/n)$.
The one-dimensionality follows from Lemma~\ref{lem:G12G13}.
\end{proof}

The two previous propositions addressed the question of whether the subperiods of a {\em planar} rhombus tilings enforce a particular slope or not.
Now, we want to determine whether subperiods enforce planarity itself:

\begin{proposition}\label{prop:n_fold_planarity}
If $5$, $7$, $8$ or $12$ divides $n$, then the subperiods of the $n$-fold tilings enforce planarity.
\end{proposition}

\begin{proof}
Let $m:=n$ if $n$ is odd, or $m:=n/2$ otherwise.
For $1\leq i\leq m$, one has
\begin{itemize}
\item For $m=4p$:
$$
G_{i,p+i}=G_{p+i,2p+i}=G_{2p+i,3p+i}=G_{i,3p+i}.
$$
\item For $m=5p$:
$$
G_{i,p+i}=G_{p+i,2p+i}=G_{2p+i,3p+i}
\qquad
G_{i,2p+i}=G_{3p+i,i}.
$$
\item For $m=6p$:
$$
G_{i,p+i}=G_{p+i,2p+i}=G_{2p+i,3p+i}
\qquad
G_{i,3p+i}=2G_{i,p+i}.
$$
\item For $m=7p$:
$$
G_{i,p+i}=G_{p+i,2p+i}
\qquad
G_{i,2p+i}=G_{2p+i,4p+i}
\qquad
G_{p+i,4p+i}=G_{4p+i,i}.
$$
\end{itemize}
Since each of the above equalities involves only three different indices, it cor\-res\-ponds to a subperiod.
Moreover, these subperiods lift in the slope of the $n$-fold tiling onto pairwise non-collinear vectors.
Hence, if $\mathcal{T}$ is a tiling with these subperiods and $\mathcal{S}$ a lift of it, then Theorem~\ref{th:planarity_codim2} yields the planarity of the projections of $\mathcal{S}$ onto the four-dimensional space (indices are taken modulo $m$)
$$
R_i:=\mathbb{R}\vec{e}_i+\mathbb{R}\vec{e}_{p+i}+\mathbb{R}\vec{e}_{2p+i}+\mathbb{R}\vec{e}_{ap+i},
$$
where $a=4$ for $m=7p$ or $a=3$ for $m\in\{4p,5p,6p\}$.
Let $S_i$ be the slope of the projection of $\mathcal{S}$ onto $R_i$.
Let also $\widetilde{S}_i$ be the $(m-2)$-dimensional vectorial space whose projection onto $R_i$ is $S_i$.
The lift $\mathcal{S}$ stays at bounded distance from the vectorial space $V:=\widetilde{S}_1\cap\ldots\cap\widetilde{S}_m$.
We shall show $\dim V=2$ to prove the planarity of $\mathcal{S}$.
The slope $S_i$ turns out to be generated by the vectors
$$
\vec{u}_i:=(-G_{i,p+i},0,G_{p+i,2p+i},G_{p+i,ap+i})
\quad\textrm{and}\quad
\vec{v}_i:=(0,G_{i,p+i},G_{i,2p+i},G_{i,ap+i}).
$$
Consider $\vec{x}=(x_1,\ldots,x_m)\in V$.
For each $i$, there are reals $\lambda_i$ and $\mu_i$ such that the projection of $\vec{x}$ onto $R_i$ writes $\lambda_i\vec{u}_i+\mu_i\vec{v}_i$.
In particular, this yields
\begin{eqnarray*}
x_i &=& -\lambda_i G_{i,p+i},\\
x_{p+i} &=& \mu_i G_{i,p+i} = -\lambda_{p+i} G_{p+i,2p+i},\\
x_{2p+i} &=& \lambda_i G_{p+i,2p+i} +\mu_i G_{i,2p+i} =  \mu_{p+i} G_{p+i,2p+i} = -\lambda_{2p+i} G_{2p+i,ap+i},
\end{eqnarray*}
whence the second order recurrence relation on the $\lambda_i$'s:
$$
-G_{2p+i,3p+i}\lambda_{2p+i}=-G_{i,2p+i}\lambda_{p+i}+G_{i,p+i}\lambda_i.
$$
Once two of the $\lambda_i$'s are fixed, all the other ones are thus uniquely determined, and then the $\mu_i$'s by $\mu_i G_{i,p+i} = -\lambda_{p+i} G_{p+i,2p+i}$.
This proves $\dim V=2$.
\end{proof}

In particular, if $n$ is an odd multiple of $5$ or $7$, then Prop. \ref{prop:subperiod_rules} ensures that one can enforce by local rules subperiods which in turn enforce planarity by Prop.~\ref{prop:n_fold_planarity} (because $n$ is a multiple of $5$ or $7$), and the slope by Prop.~\ref{prop:n_fold_dim0} (because $n$ is odd):

\begin{corollary}\label{cor:local_rules_n_fold}
The $n$-fold tilings admit local rules for odd $n$ multiple of $5$ or $7$.
\end{corollary}

Actually, it is known that $n$-fold tilings admit local rules for any $n$ which is not a multiple of $4$ \cite{socolar}.
Here, we managed to show that there are local rules which enforce the $n$-fold symmetry of a planar tiling for any $n$ which is not a multiple of $4$ (Prop.~\ref{prop:n_fold_dim0}), but our general method failed to show that local rules can also enforce planarity in all these cases (Prop.~\ref{prop:n_fold_planarity}).
This is because we relied on a codimension two result (Th.~\ref{th:planarity_codim2}) whereas a general characterization of planarity in any codimension remains to be found (which should in particular apply to $n$-fold tilings).
Nevertheless, Propositions~\ref{prop:n_fold_dim1} and \ref{prop:n_fold_planarity} deal with cases that are not considered in \cite{socolar}, namely $n$-fold tilings when $n$ is a multiple of $8$ or $12$.
Indeed, these propositions allow to show that the slope of such $n$-fold tilings can be obtained as the solution of a simple optimization problem, reminding discussions about {\em optimal cluster covering} for non-periodic tilings and what is sometimes referred to as {\em maxing rules} (see, {\em e.g}, \cite{maxing1,maxing2,maxing3}):

\begin{proposition}\label{prop:even_fold_minimization}
When $n$ is a multiple of $8$ or $12$, there are local rules such that the $n$-fold tilings are the tilings satisfying these local rules and minimizing the proportion of tiles $T_{i,i+2}$ (for all $i$).
\end{proposition}

\begin{proof}
The planarity (and thus the existence of tile frequencies) is ensured by Prop.~\ref{prop:n_fold_planarity}.
In Prop.~\ref{prop:n_fold_dim1}, we saw that in such a case the product $XY$ is constant, where $X=G_{i,i+2}$ for odd $i$ and $Y=G_{i,i+2}$ for even $i$.
The $n$-fold tilings correspond to $X=Y$.
This happens for $X+Y$ minimal, that is, according to Prop.~\ref{prop:frequences}, for the minimal proportion of tiles $T_{i,i+2}$.
 \end{proof}

\begin{figure}[hbtp]
\includegraphics[width=0.31\textwidth]{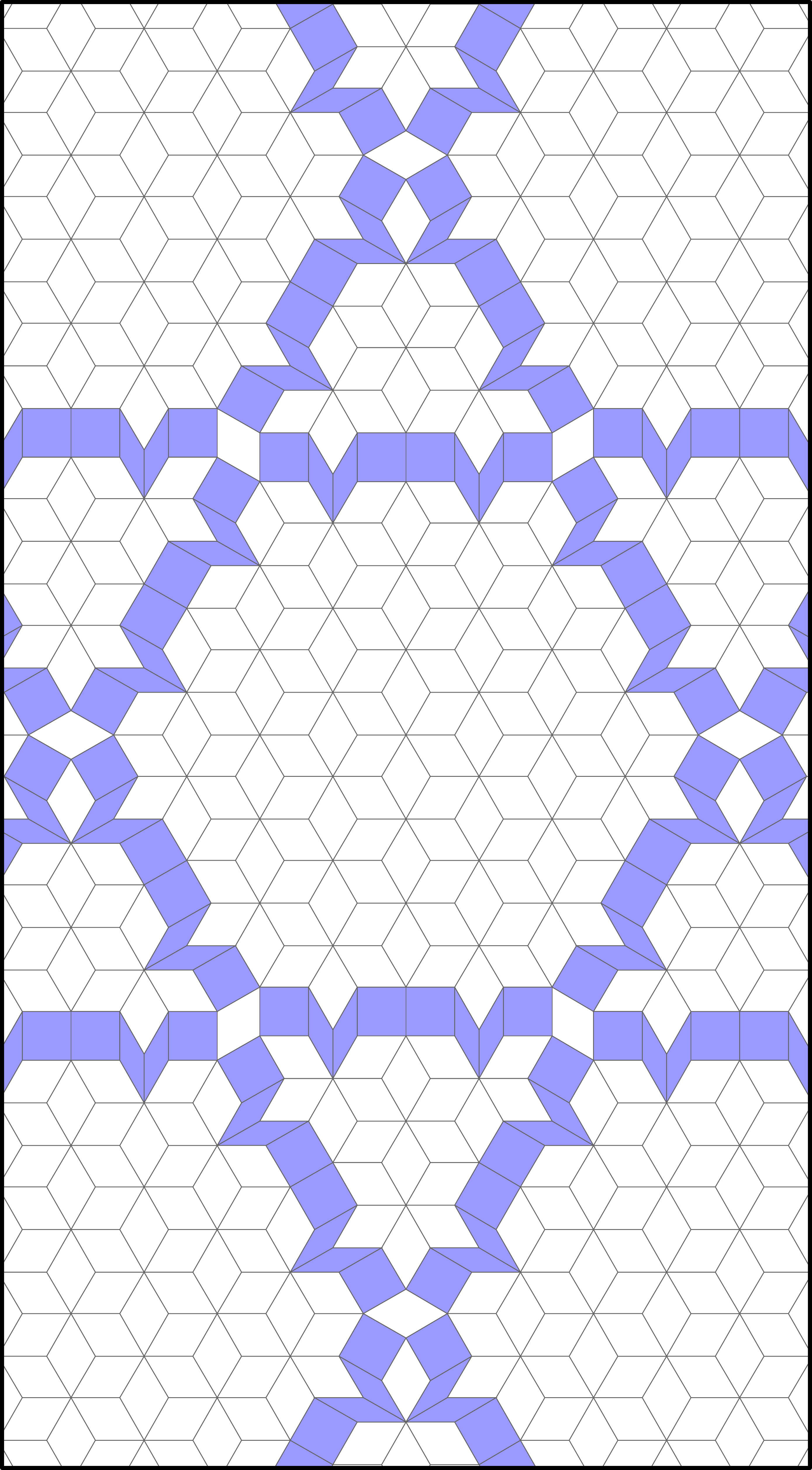}
\hfill
\includegraphics[width=0.31\textwidth]{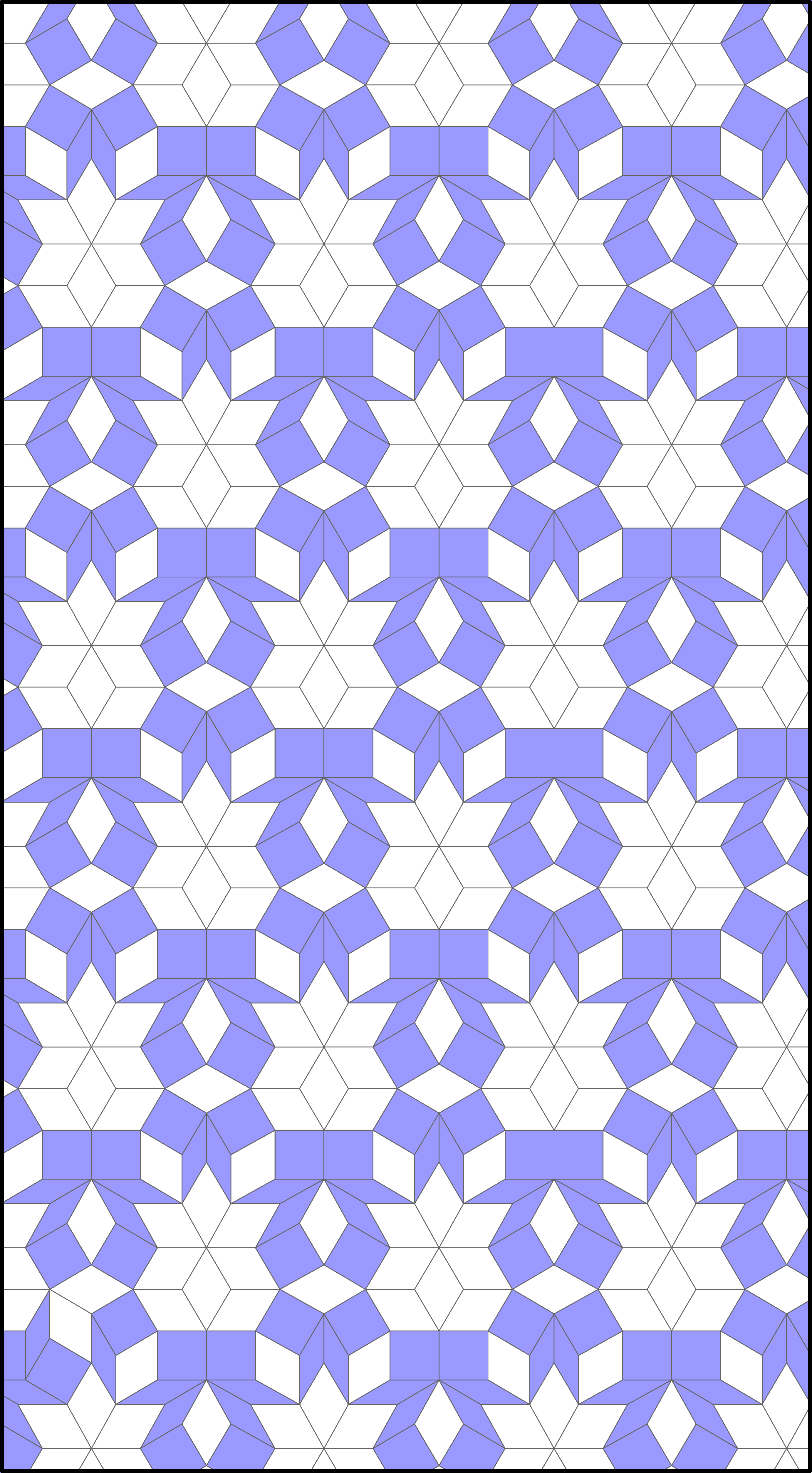}
\hfill
\includegraphics[width=0.31\textwidth]{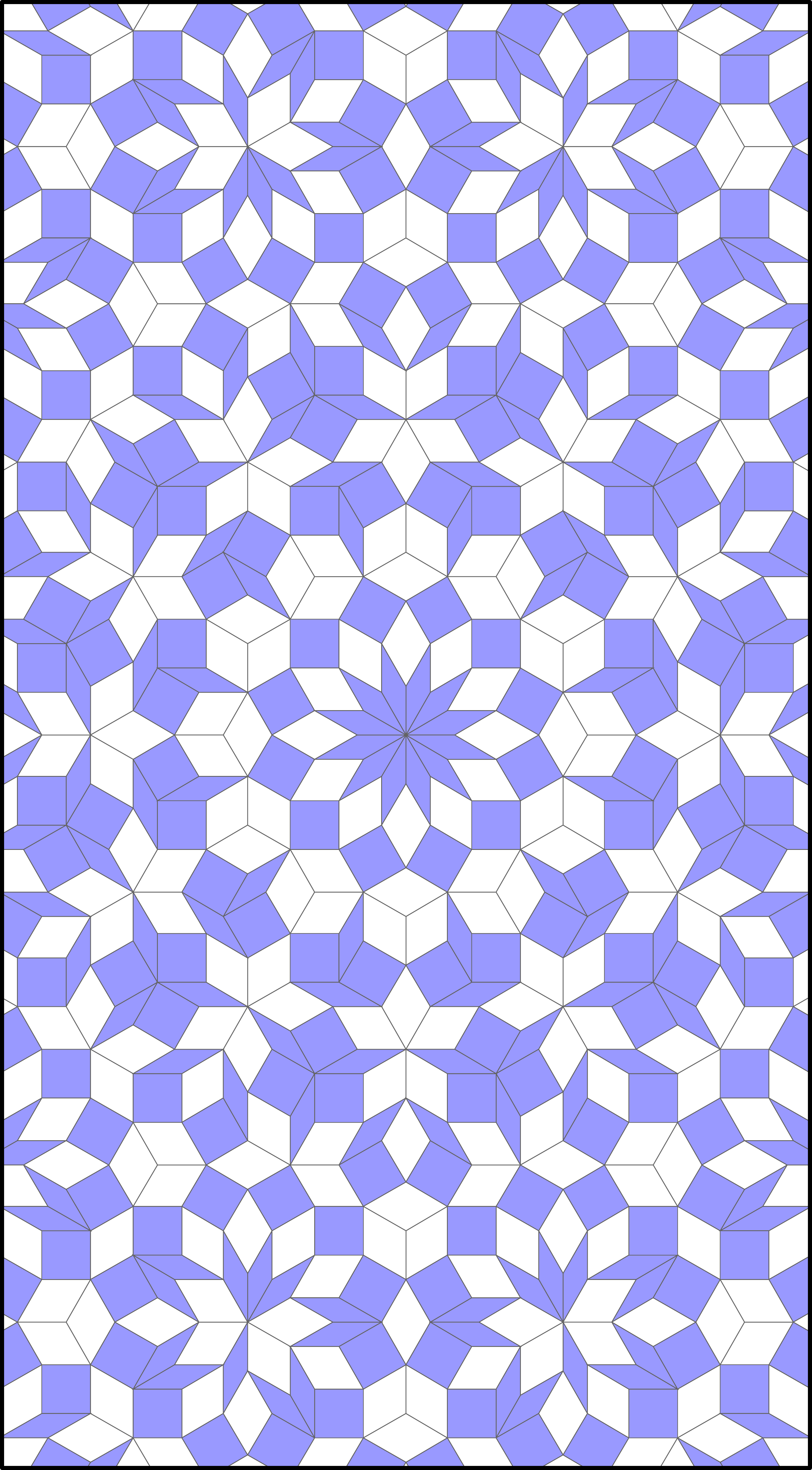}
\caption{
The $12$-fold tilings (rightmost) has subperiods which enforce planarity but allow a one-parameter family of tilings ({\em e.g.} the leftmost and central ones).
Among these tilings, the $12$-fold minimize the proportion of white tiles.
}
\label{fig:dodecagonal}
\end{figure}

We retrieve the fact that Ammann-Beenker $8$-fold tilings are characterized by subperiods and minimization of the proportion of square tiles (end of Section~\ref{sec:ammann_beenker}).
Fig.~\ref{fig:dodecagonal} depicts the $12$-fold case.\\

Actually, one can see in the proof of Prop.~\ref{prop:n_fold_planarity} that only a subset of the subperiods are used to prove the planarity of the $n$-fold tilings.
Hence, if we define local rules that enforce these subperiods but not necessarily those used to prove the zero- or one-dimensionality of the system (Prop.~\ref{prop:n_fold_dim0} or \ref{prop:n_fold_dim1}), then we can get local rules allowing a many-parameters family of planar rhombus tilings, among which $n$-fold tilings satisfy a similar optimization problem.
This could be meaningful in the context of quasicrystal modelization if we assume that complicated local rules means rather implausible atom arrangements, while minimizing tile proportions simply means playing with molecular concentrations.

%%%%%%%%%%%%%%%%%%%%%%%%%%%%%%%%%%%%%%

\end{document}